\documentclass[12pt,centertags,oneside]{amsart}
\usepackage{amsmath,amstext,amsthm,amscd,typearea,hyperref,mathtools}
\usepackage{amssymb}
\usepackage{a4wide}
\usepackage[mathscr]{eucal}
\usepackage{mathrsfs}
\usepackage{typearea}
\usepackage{charter}
\usepackage{pdfsync}
\usepackage[a4paper,width=16.2cm,top=3cm,bottom=3cm]{geometry}

\numberwithin{equation}{section}

\usepackage{xcolor}

\usepackage{tikz}

\newtheorem{theorem}{Theorem}[section]
\newtheorem{proposition}[theorem]{Proposition}
\newtheorem{corollary}[theorem]{Corollary}
\newtheorem{lemma}[theorem]{Lemma}
\newtheorem{remark}[theorem]{Remark}

\theoremstyle{definition}
\newtheorem{definition}[theorem]{Definition}

\newtheorem*{theorem*}{Theorem}

\newcommand{\cali}[1]{\mathscr{#1}}

\newcommand{\Leb}{{\rm Leb}}

\newcommand{\supp}{{\rm supp }\,}

\newcommand{\Lip}{{\rm Lip}}
\newcommand{\vect}{{\rm Vect }\,}

\newcommand{\Grm}{{\rm G}}
\newcommand{\Brm}{{\rm B}}

\newcommand{\Crm}{{\rm C}}

\newcommand{\dist}{\mathop{\mathrm{dist}}\nolimits}

\newcommand{\Area}{\mathop{\mathrm{Area}}\nolimits}

\newcommand{\ddc}{{\rm dd}^c}

\newcommand{\omegaFS}{ \omega_{\mathrm{FS}}}
\newcommand{\Tan}{{\rm Tan}}

\def\u{\operatorname{u}}
\def\s{\operatorname{s}}

\usepackage{mathtools}
\DeclarePairedDelimiter\ceil{\lceil}{\rceil}

\newcommand{\Ac}{\cali{A}}
\newcommand{\Bc}{\cali{B}}
\newcommand{\Cc}{\cali{C}}

\newcommand{\Ec}{\cali{E}}

\newcommand{\Gc}{\cali{G}}

\newcommand{\Lc}{\cali{L}}
\newcommand{\Nc}{\cali{N}}

\newcommand{\Qc}{\cali{Q}}

\newcommand{\Sc}{\cali{S}}

\newcommand{\B}{\mathbb{B}}
\newcommand{\D}{\mathbb{D}}
\newcommand{\C}{\mathbb{C}}

\newcommand{\Z}{\mathbb{Z}}
\newcommand{\R}{\mathbb{R}}

\renewcommand\P{\mathbb{P}}

\newcommand{\W}{\mathbb{W}}

\newcommand{\T}{\mathbb{T}}

\newcommand{\diam}{{\rm diam}}

\date{ \today}

\title[ ]{Quantitative Dynamics of complex Hénon maps}

\begin{author}[H. de Th\'elin]{Henry de Th\'elin}
\address{Universit\'e Paris 13, Sorbonne Paris Nord, LAGA, CNRS (UMR 7539),
F-93430, Villetaneuse, France. }
\email{dethelin@math.univ-paris13.fr}
\end{author}

\begin{author}[T.-C. Dinh]{Tien-Cuong Dinh}
\address{National University of Singapore, Lower Kent Ridge Road 10,
Singapore 119076, Singapore}
\email{matdtc$@$nus.edu.sg }
\end{author}

\begin{author}{Lucas Kaufmann}
\address{Université d’Orléans, Université de Tours, CNRS, IDP UMR 7013, 45067 Orléans, France.}
\email{lucas.kaufmann@univ-orleans.fr}
\end{author}


\allowdisplaybreaks
\emergencystretch=\maxdimen
\begin{document}

\begin{abstract}
Let $f$ be a Hénon map of $\C^2$.  We provide  quantitative versions of several results involving dynamical objects associated to $f$.  Our results can be interpreted as a quantitative version of Pesin theory with geometric control. As an application we show that the perdiodic points of period $n$ of $f$ equidistribute towards its equilibrium measure exponentially fast as $n$ tends to infinity
\end{abstract}

\maketitle

\tableofcontents

\section{Introduction}

Hénon maps form a  widely studied class of invertible dynamical systems in dimension two.  The dynamical theory of such maps is highly developed, thanks to many deep results obtained by numerous authors.  Rather than reviewing the complete literature,   we refer to \cite{henon:survey} and the references therein for a historical overview. For complex Hénon maps,  a fundamental result of Bedford-Lyubich-Smillie \cite{bedford-lyubich-smillie} asserts that the periodic points of $f$ equidistribute towards its unique measure of maximal entropy. This result  can be viewed as the culmination of a series of previous works and the proof combines pluripotential theory and Pesin theory.   

The main goal of the present work is to furnish quantitative versions of many of the known results in the dynamical theory of complex Hénon maps.  Broadly speaking,  our goal is to improve the standard results by promoting the usual small quantities by \textit{exponentially small} ones. This include the  laminarity of Green currents,  their geometric intersection and graph transforms.  Our results can be viewed as a quantitative version of Pesin theory.  As an application,  we are able to quantify the above mentioned equidistribution result of Bedford-Lyubich-Smillie.  See Theorem \ref{thm:periodic-points} below for a more precise statement.

\begin{theorem*}
Let $f: \C^2 \to \C^2$ be a Hénon map of degree $d \geq 2$.    Denote by $\mu$ its measure of maximal entropy.  Then, the saddle periodic points of period $n$ of $f$ equidistribute towards $\mu$ exponentially fast as $n$ tends to infinity.
\end{theorem*}

\medskip

\textbf{History and non-AI usage.}
This project began in autumn 2025 as a continuation our previous collaboration \cite{ddk:periodic-Pk}.  Most of the mathematical work was finished by March 2026, and since then, the writing process has been taking place intermittently.    No AI was used whatsoever in this process.  In light of the current surge in AI-generated or AI-assisted proofs, we have decided to release the current version in its present form to avoid any claims that could question the originality and relevance of our work. A final version will be updated soon.

\medskip

\textbf{Convention.} Throughout this article,  $C, C_1, C_2$, etc will denote a positive constant whose precise value is irrelevant and may change from line to line.  The dependence of these constants on certain parameters (or
lack thereof), if not explicitly stated, will be clear from the context. The important thing is that they  are all independent of the relevant dynamical and geometric parameters ($n,  r$, etc.).

\section{Green currents and the equilibrium measure}

 Let $f: \C^2 \to \C^2$ be a Hénon map of degree $d \geq 2$.  We also denote by $f: \P^2  \dashrightarrow \P^2$ its extension as a birational map to $\P^2$.  Denote by $I^+$ (resp.  $I^-$) the indeterminacy set of $f$ (resp. $f^{-1}$).  Up to an affine change of coordinates of $\C^2$,  we will assume that $I^+ = [1:2:0]$ and $I^- = [2:1:0]$.  Notice that the two families of lines parallel to the coordinate axis in $\C^2$ extend to two pencil of lines in $\P^2$ going through $a^+ := [0:1:0]$ and $a^-:= [1:0:0]$ respectively.

The Green functions are defined by
$$
G^+(z) = \lim_{n \to \infty} \frac{1}{d^n} \log^+ \|f^n(z)\| \quad \text{and} \quad
G^-(z) = \lim_{n \to \infty} \frac{1}{d^n} \log^+ \|f^{-n}(z)\|,
$$

The stable current is by definition $T^+ := \ddc G^+$. If $L \subset \P^2$ is a projective line avoiding $I^+$, then $$\lim_{m \to \infty} \frac{1}{d^m} (f^m)^*[L] = T^+.$$

Similarly, one defines the unstable current by $T^- := \ddc G^-$. If $L \subset \P^2$ is a projective line avoiding $I^-$, then $$\lim_{m \to \infty} \frac{1}{d^m} (f^m)_*[L] = T^-.$$

The equilibrium measure is the probability measure defined by
$$\mu: T^+ \wedge T^-.$$

This is a canonical probability measure that is invariant under $f$.  We will use in a crucial way the laminar structure of the currents and, in particular,  the geometric nature of their intersection. We refer to \cite{dujardin:laminar-P2} and \cite{dujardin:intersection} for the  basic definitions and some fundamental results in this topic.

\section{Dynamical currents} \label{sec:geometric-int}

Our results will be based on the construction of several currents of dynamical nature  with nice geometry and good compatibility with the hyperbolic nature of the dynamical system induced by $f$. We start by defining a suitable partition of the dynamical plane and by introducing some key parameters.

\begin{definition} \label{def:manhattan}
We identify $\C$ with $\R^2$ and $\C^2$ with $\R^4$.  Denote by $\W$ (resp.  $\W^2$) the open square $(-1,1)^2$ in $\R^2$ (resp.  the open cube $(-1,1)^2$ in $\R^4$).  Fix $0<r<1$ and let $\tau = (\tau_1, \tau_2) \in \R^2 \times \R^2$.   For $\eta = (\eta_1,\eta_2) \in \Z^2 \times \Z^2$ we will denote
$$\W_{r,\eta_i}:=r\W + \tau_i + 2r \eta_i  \quad \text{for} \quad i=1,2$$
and 
$$\W_{r,\eta}^2:=  \W_{r,\eta_1} \times \W_{r,\eta_2} = r\W^2 + \tau + 2r \eta.$$ 

We will call the $\W_{r,\eta_i}$ squares and the $\W_{r,\eta}^2$ cubes or cells.  Observe that the union of the closed cells $\overline{\W}_{r,\eta}^2$ for  $ \eta\in\Z^4$ covers $\C^2$.  The point $\tau + 2r \eta$ is called the center of the cell.  We omit $\tau$ from our notation for simplicity.

For $0<\lambda<1$,  we denote by $\W_{r,\eta}^2$ the rescaling of $\W_{r,\eta}^2$ by the factor $\lambda$,  that is $$\lambda \W_{r,\eta}^2:= \lambda r\W^2 + \tau + 2r \eta$$
and we define the \emph{street network} $$\lambda \mathbb S_r := \bigcup_{\eta \in \Z^2} \W_{r,\eta}^2 \setminus \bigcup_{\eta \in \Z^2}  \lambda  \W_{r,\eta}^2.$$
\end{definition}

Let $0<\gamma^\pm \leq 1$ be the Hölder exponents of the Green functions $G^\pm$.  We fix a constant $$0 < \gamma < \frac13 \max\{\gamma^+,\gamma^-\}.$$

We fix a large $R>0$ such that $J^* \subset \D(0;R)^2$.  We will denote 
\begin{equation}
P_r : = \bigcup_{\W_{r,\eta}^2 \cap \D(0;R)^2 \neq \varnothing } \W_{r,\eta}^2 \quad \text{and} \quad (1-(2r)^\gamma) P_r : = \bigcup_{\W_{r,\eta}^2 \cap \D(0;R)^2 \neq \varnothing } (1-(2r)^\gamma) \W_{r,\eta}^2 
\end{equation} and set
\begin{equation} \label{eq:kappa-def}
    \kappa:= \log_d \max_{x \in \overline{\D(0;R)^2}} \|Df(x)\|.
\end{equation}
By taking a larger $R$ if necessary we may assume that $\kappa >1$.

\begin{lemma}[Dujardin] \label{lemma:dujardin-fubini}
Denote by $\sigma_{T^\pm}=T^\pm \wedge \omegaFS$ the trace measure of $T^\pm$. There exists a constant $c>0$ such that $$\big( \mu + \sigma^+ + \sigma^- \big)((1-(2r)^\gamma) P_r) \leq c r^\gamma.$$ 
\end{lemma}

Some of the steps of our construction work on a scale $r$ independent of $n$, but some of them require $r$ to shrink exponentially with $n$. For simplicity we will make this assumption from the beginning and we will omit the dependence on $n$ from the notation.

Let $0 < \alpha < 1$ be a small constant and $\kappa >1$ be as in \eqref{eq:kappa-def}.  We set
\begin{equation*}
    r:= d^{-\alpha n} \quad \text{and} \quad \rho:= r^\kappa.
\end{equation*}

We fix a large $n \geq 1$. Our construction of cells, currents, Pesin charts etc. will be adapted to $f^n$ and intermediate iterates.   Ultimately,  this will produce in each dynamical box a (single) fixed point of $f^n$ of saddle type and in the support of $\mu$.   We will use the fact that $n$ is large in order to absorb some universal constants.

\medskip

\textbf{Scales and projections.} 
Our arguments will involve three distinct scales, all exponentially small. In the first step of the construction, we work at the scale $r = d^{-\alpha n}$, using the two canonical projections of $\C^2$.  In a second step, after passing to a finer scale $\rho^{25}$, we show that the stable and unstable discs obtained are graphs with respect to a single projection $\pi_D : \C^2 \to D$, for a suitable line $D$, together with a good control on their slopes.  In the final part of the argument, where a graph transform is used to produce periodic points as intersections of stable and unstable discs, we will work inside bidiscs of size $\rho^{60}$ by $\rho^{100}$.

\medskip

We introduce the following non-negative integers    $p$, $\ell$ and $0 \leq s \leq p$
\begin{equation}
    p:= \ceil{6 \alpha n} \quad \text{and} \quad  n= (\ell +1) p  + s.
\end{equation}

We will construct currents $$T^\pm_r:= T^\pm_{r,0}, \quad T^\pm_{r,-p},\, \ldots\, , T^\pm_{r,-\ell p}, \, T^\pm_{r,-n} \quad \text{and} \quad T^\pm_{r,p}$$ with the following properties:
\smallskip

\begin{enumerate}
    \item[(P1)] The $T^\pm_{r,\bullet}$ are positive  currents of bi-degree $(1,1)$ and $T^\pm_{r,\bullet}\leq T^\pm$
    
    \item[(P2)] The currents $T^\pm_{r,\bullet}$ are uniformly laminar over the union of  $(1-(2r)^\gamma) \W_{2r,\eta}^2$,  laminated by holomorphic discs (called \textit{plaques}) that are graphs over $\W_{2r,\eta_1}$ or $\W_{2r,\eta_2}$.
    
    \item [(P3)] The geometric intersections $\overline \mu_{r,\bullet}:= T^+_{r,\bullet} \dot{\wedge} T^-_{r,\bullet}$ are dominated by $\mu$ and $\mu - \overline \mu_{r,\bullet}$ is of mass at most $C r^{\gamma}$. 
    
       \item[(P4)] The plaques of $T^-_{r,\bullet}$ are compatible in the following sense: for $j=-1,\ldots,\ell$,  if $\Delta^{\u}_{-jp}$ is a plaque of $T^-_{r,-jp}$ and $\Delta^{\u}_{-jp+p}$ is a plaque of $T^-_{r,-jp+p}$  then $f^p(\Delta^{\u}_{-jp})$  is either disjoint from $\Delta^{\u}_{-jp+p}$  or they meet along a common open set.   An analogous property holds for the plaques of $T^-_{r,\-n}$ and $T^-_{r,-\ell p}$ under the action of $f^{p+s}$.
    
   \item[(P5)]  Let $\Delta^{\s}$ be a plaque of $ T^+_{r,-jp}$ for some $j=-1, 0,1,\ldots, \ell$.  By (P2),  it is a graph over $\W_{2r,\eta_i}$ for $i=1$ or $2$.   Denote by $\widetilde{\Delta^{\s}} \subset \Delta^{\s}$ the corresponding graph over $\W_{r,\eta_i}$.  If $x \in \widetilde{\Delta^{\s}}$ and $v^{\s}_x$ is a unit vector tangent to $\widetilde{\Delta^{\s}}$ at $x$, then 
   \begin{equation} \label{eq:stable-estimate}
   \|D f^p(x)v^{\s}_x\|  \leq   d^{-\alpha n/4} . 
   \end{equation}
    and similarly for $ T^+_{r,-n}$ by replacing $p$ by $p+s$. 
   \item[(P6)]  Let $\Delta^{\u}$ be a plaque of $ T^-_{r,-jp}$ for some $j=-1, 0,1,\ldots, \ell$.  By (P2),  it is a graph over $\W_{2r,\eta_i}$ for $i=1$ or $2$.   Denote by $\widetilde{\Delta^{\u}} \subset \Delta^{\u}$ the corresponding graph over $\W_{r,\eta_i}$.  If $x \in \widetilde{\Delta^{\u}}$ and $v^{\u}_x$ is a unit vector tangent to $\widetilde{\Delta^{\u}}$ at $x$, then 
   \begin{equation} \label{eq:unstable-estimate}
   \| D f^{-p}(x)v^{\u}_x\|  \leq  d^{-\alpha n/4}. 
   \end{equation}
     and similarly for $ T^-_{r,-n}$ by replacing $p$ by $p+s$.

    \item[(P7)]Denoting by $\overline \Bc_\bullet:= \supp \overline \mu_{r,\bullet}$ and $$\overline{\Bc}^*:=   f^{-p}(\overline \Bc_{p}) \cap \overline \Bc_0 \cap f^p(\overline \Bc_{-p}) \cap \cdots \cap f^{\ell p}(\overline \Bc_{-\ell p}) \cap f^n(\overline \Bc_{-n}),$$
    we have that $\mu_r(\overline{\Bc}^*) \geq 1 - Cr^\gamma$ for some constant $C>0$.  Here $\mu_r = \mu_{r,0}$.
    
   \item[(P8)] If $x \in \Bc^*$ and $\Delta^{\u}_{-jp}$ (resp.~ $\Delta^{\s}_{-jp}$),  $j=-1,0,1,\ldots,\ell$ is the plaque of $T^-_{r,-jp}$ (resp.~ $T^+_{r,-jp}$)  through $x$,  then  $\measuredangle (\Delta^{\s}_{-jp},\Delta^{\u}_{-jp}) > \rho^{7}$.
  \end{enumerate}

\medskip
In a second step, we will refine the above construction to obtain $\T^\pm_{r,\bullet}$ currents with the following properties:
\smallskip

\begin{enumerate}
    \item[(P9)] There exists a line $D \subset \C^2$,  an orthogonal projection $\pi_D: \C^2 \to D$ and uniformly laminar currents $\T^\pm_{r,\bullet} \leq T^\pm_{r,\bullet}$ such that the plaques of $\mathbb T^\pm_{r,\bullet}$ are holomorphic graphs relative to $\pi_D$ above squares of side $\rho^{25}$ inside $D$ whose slope is bounded from above by $1/\rho^{10}$. 
    
    \item[(P10)] Let $\mu_{r,\bullet}:= \T^+_{r,\bullet} \dot{\wedge} \T^-_{r,\bullet}$. Then,  Property (P3) holds for $\mu_{r,\bullet}$ instead of $\overline \mu_{r,\bullet}$.
    
   \item[(P11)] The estimates in Property  (P7)  hold after replacing $\overline \mu_{r,\bullet}$ by $\mu_{r,\bullet}$.
    
\end{enumerate}

\begin{remark} \label{rmk:inherit}
Observe that,  since $\T^\pm_{r,\bullet} \leq T^\pm_{r,\bullet}$,  every plaque of $\T^\pm_{r,\bullet}$ is also a plaque of   $T^\pm_{r,\bullet}$.  In particular,  property (P4) and the estimates \eqref{eq:stable-estimate} and \eqref{eq:unstable-estimate} also hold for the plaques of $\T^\pm_{r,\bullet}$
\end{remark}

\subsection{Definition of $T^\pm_{r,\bullet}$  and properties (P1) to (P4)} \label{subsec:P1-P4}

We now begin the construction of the currents described above.  We keep the same notations.  We will construct the currents $T^+_{r,\bullet}$ in full detail.  The currents $T^-_{r,\bullet}$ can be built  in an analogous way by replacing $f$ by $f^{-1}$.

Recall that $n \geq 1$ is fixed and sufficiently large. The starting point is the convergence $\lim_{m \to \infty} \frac{1}{d^m} (f^m)^*[L] = T^+.$ Taking $m \gg n$, we see that $T^+$ is well approximated by $d^{-m}$ times the current of integration along the algebraic curve $f^{-m}(L)$. We will restrict this curve to each cell $\W_{r,\eta}^2$ and discard the components with undesirable properties (e.g., not graphs,  excessively large area under the action of iterates of $f$,   etc.). This will give the desired currents.

\medskip

From here on, we fix a line $L \subset \P^2$ not passing through $I^+$ and $I^-$.

\begin{definition}[The set of good plaques $\Grm^{i,+}_{r,m}$] \label{def:good-plaques1}
Let $\pi_i: \C^2 \to \C$, $i=1,2$ be the two canonical projections.  Fix $m \geq 1$.  We denote by $\Grm^{i,+}_{r,m}$ the set of all connected components $\Delta$ of $\pi_i^{-1}(\overline \W_{2r,\eta_i}) \cap f^{-m}(L)$ such that $\pi_i\big|_\Delta$ is a biholomorphism onto its image  and such that $\Area(f^j(\Delta)) \leq d^{-j/2}$ for all $0 \leq j < m$.
\end{definition}

\begin{lemma}  \label{lemma:slope0}
Let $\Delta \in \Grm^{1,+}_{r,m}$ and write it  as $\Delta = \{(z,\varphi(z)): z \in \W_{2r,\eta_1}\}$ for some holomorphic function $\varphi: \W_{2r,\eta_1} \to \C$.  There exists a constant $c_0>0$ independent of $r$ and $N$ such that  $\| \varphi' \|_{L^{\infty}( \W_{r,\eta_1})} \leq c_0 r^{-1}$ and  $\| \varphi'' \|_{L^{\infty}( \W_{r,\eta_1})} \leq c_0 r^{-2}$.  An analogous result holds for $\Delta \in \Grm^{2,+}_{r,m}$ given as the graph above the second coordinate axis.
\end{lemma}

\begin{proof}
Fix  $\Delta$ as in the statement and set $$\widetilde{\Delta}  := \{(z,\varphi(z)): z \in \W_{r,\eta_1}\} \quad \text{and} \quad \widehat{\Delta}   := \{(z,\varphi(z)): z \in \W_{3r/2,\eta_1}\},$$ so that $\widetilde{\Delta} \subset \widehat{\Delta} \subset \Delta$.  By definition, we have that  $\Area(\Delta) \leq 1$.  Let $M$ be the modulus of $\Delta \setminus \widehat{\Delta}$ and observe that it is independent of $r$ and $N$.   By \cite[Appendice]{briend-duval:ihes},  there exists a constant $c>0$ such that $\diam (\widehat{\Delta})^2 \leq c M^{-1} \Area((\Delta)) \leq c M^{-1 }$, so $ \diam (\widehat{\Delta})  \leq c^{1/2}M^{-1/2 }$.  In particular,  if $p_\eta$ denotes the center of  $\W_{r,\eta_1}$,  the function $\varphi(z) - \varphi(p_\eta)$ is bounded over $\W_{3r/2,\eta_1}$  by a constant independent of $r$.  A standard application of Cauchy estimates for $\varphi'$ and $\varphi''$ over the disc $\W_{r,\eta_1}$ yields the desired result.
\end{proof}

For $i=1,2$ we define the currents
\begin{equation*}
S^{i,+}_{r,m} = \frac{1}{d^m} \sum_{\Delta \in \Grm^{i,+}_{r,m}} [\Delta].
\end{equation*}

Observe that, by construction,  $0\leq  S^{i,+}_{r,m}  \leq d^{-m} [ f^{-m}(L)]$.  Note however that $S^{i,+}_{r,m} $ is not closed.  Its boundary is contained in union of the boundaries of the cells $\overline \W_{2r,\eta}^2$.

\begin{proposition} \label{prop:dujardin}
There is a constant $C>0$ independent of $r$ and $m$ such that
$$\langle d^{-m} [f^{-m}(L)]  -  S^{1,+}_{r,m},  i dz \wedge d \overline z \big|_{\D(0,2R)} \rangle \leq Cr^2.$$
and
$$\langle d^{-m} [f^{-m}(L)]  -  S^{2,+}_{r,m},  i dw \wedge d \overline w \big|_{\D(0,2R)} \rangle \leq Cr^2.$$
\end{proposition}

\begin{proof}
The result is a refinement of an estimate of Dujardin.   We will only prove the first inequality,  the second being analogous.  In order to lighten the notation,  set $S:= S^{1,+}_{r,m}$ and $\Grm:=  \Grm^{i,+}_{r,m}$. We will introduce an intermediary current and estimate its difference with  $S$.  Let $\Grm'$ be the set the set of all connected components $\Delta$ of $\pi_1^{-1}(\overline \W_{2r,\eta_1}) \cap f^{-m}(L)$ such that $\pi_i\big|_\Delta$ is a biholomorphism onto its image. In other words,  the components of $\Grm'$   satisfy only the first condition  of Definition \ref{def:good-plaques1}.  In particular $\Grm   \subset \Grm'$. Set $S':= d^{-m} \sum_{\Delta \in \Grm'} [\Delta]$. Notice that $S  \leq S' \leq d^{-m} [f^{-m}(L)]$. 

 The  estimate  in  \cite[p. 753]{dujardin:laminar-P2}  implies that  $\langle d^{-m} [f^{-m}(L)]  -  S',  i dz \wedge d \overline z \big|_{\D(0,2R)} \rangle \leq Cr^2$,  so it is enough to prove that
\begin{equation} \label{eq:S'-S}
 \langle  S' - S ,   i dz \wedge d \overline z \big|_{\D(0,2R)} \rangle \leq C' r^2
\end{equation}
for some constant $C' > 0$.
By the definition of the above currents we have that $S' - S  = d^{-m} \sum_{\Delta \in \Grm' \setminus \Grm} [\Delta]$ and the plaques $\Delta \in \Grm' \setminus \Grm$  satisfy $\Area(f^{j_0}(\Delta)) > d^{-j_0/2}$ for some $0 \leq j_0 < m$.  Let $\Brm_j$ be the set  of $\Delta \in \Grm' \setminus \Grm$  such that $\Area(f^j(\Delta)) > d^{-j/2}$ .  Then $$S' - S  \leq d^{-m} \sum_{j=0}^{m-1} \sum_{\Delta \in \Brm_j} [\Delta] = \sum_{j=0}^{m-1} U_j,  \quad \text{where} \quad U_j: =  d^{-m}\sum_{\Delta \in \Brm_j} [\Delta].$$ 

Fix  $0 \leq j < m$.  By construction $U_j \leq d^{-m}[f^{-m}(L)]$,  so $(f^j)_*(U_j) \leq d^{-m}[f^{-m +j}(L)]$.  In particular,  the mass of $(f^j)_*U_j$ is at most $d^{-m} d^{m-j} = d^{-j}$.  On the other hand,  the mass of $(f^j)_*U_j$ equals $d^{-m} \sum_{\Delta \in \Brm_j} \Area(f^j(\Delta)) \geq d^{-m} d^{-j/2} |\Brm_j|$.  It follows that $|\Brm_j| \leq d^m d^{-j/2}$.  Since each $\Delta$ is a graph over $\overline \W_{2r,\eta_1}$,  we have $\langle  [\Delta] ,   i dz \wedge d \overline z \rangle = 4\pi r^2$,  so $\langle  U_j ,   i dz \wedge d \overline z \big|_{\D(0,2R)} \rangle \leq  d^{-m} |\Brm_j| (\pi r^2) \leq 4 \pi r^2 d^{-j/2}$.    Combined with the above,  one concludes that
\begin{align*}
 \langle  S' - S ,   i dz \wedge d \overline z \big|_{\D(0,2R)} \rangle \leq  \sum_{j=0}^{m-1}  \langle  U_j ,   i dz \wedge d \overline z \big|_{\D(0,2R)} \rangle \leq 4 \pi r^2  \sum_{j=0}^{m-1} d^{-j/2} \leq C' r^2,
\end{align*}
where $C':= 4 \pi \sum_{j=0}^{\infty} d^{-j/2}$.  This concludes the proof.
\end{proof}

Denote by $C^{i,+}_{r,m}$ the union of all $\Delta  \in \Grm^{i,+}_{r,m}$ included in $ \D(0;R)^2$ so that the restriction of $S^{i,+}_{r,m}$ to $ \D(0;R)^2$  is given by $\frac{1}{d^m} [C^{i,+}_{r,m}]$ and set 
$$T^{+}_{r,m} = \frac{1}{d^m} [C^{1,+}_{r,m} \cup C^{2,+}_{r,m} ].$$

By construction $T^{+}_{r,m}$ is uniformly laminar,  $0 \leq S^{i,+}_{r,m} \leq T^{+}_{r,m} \leq d^{-m} [f^{-m}(L)] $ and by Proposition \ref{prop:dujardin} we have the mass estimate
$$ \Big\| d^{-m} [ f^{-m}(L)] - T^{i,+}_{r,m} \Big\|_{P_r} \leq 2Cr^2.$$

In order to obtain a compact family of plaques,  we consider only those entering a smaller subcell of $\W_{r,\eta}^2$

\begin{definition} \label{def:Gr}
We define  $\Gc_r$ as the set of analytic subsets $\Delta$ of the cells $\W_{2r,\eta}^2$ of $P_{2r}$ such that $\Delta \cap  (1-(2r)^\gamma) \overline \W_{2r,\eta}^2 \neq \varnothing$ and  $\Delta$ is included in a holomorphic disc $\Delta' \subset \overline \W_{2r,\eta}^2$ whose area is less than $k$ and such that either $\pi_1\big|_{\Delta' \cap \W_{2r,\eta}^2 }$ or  $\pi_2\big|_{\Delta' \cap \W_{2r,\eta}^2 }$ is biholomorphic.
\end{definition}

The above conditions implies that the area of the graphs $\Delta$ are uniformly bounded,  so by Bishop's theorem  $\Gc_r$ equipped with the Hausdorff distance is a compact space. In particular, the space of positive measures on $\Gc_r$ whose mass is uniformly bounded by a given constant is compact.  A theorem of Lelong implies that the graphs $\Delta$ in the above definition have area at least $(2r)^{2+2\gamma}$.  

Observe that we can write $$\mathbf 1_{(1-(2r)^\gamma) P_{2r}} T^{+}_{r,m} = \int_{\Gc_r} [\Delta \cap (1-(2r)^\gamma) P_r] d \nu^+_{r,m}(\Delta)$$ for some positive measure $\nu^+_{r,m}$ on the space $\Gc_r$.   From the fact that the mass of $T^{i,+}_{r,m} $ is bounded by one and  the area of $\Delta$ is at least $(2r)^{2+2\gamma}$,  it follows that the mass of the measures $\nu^+_{r,m}$ are bounded by a constant independent of $m$.  In particular $\{\nu^+_{r,m}\}_{m \geq 1}$ is a compact family of measures on  $\Gc_r$.  

\begin{definition} \label{def:T^+_r}
Let $\nu^+_{r} = \lim_{l \to \infty} \nu^+_{r,m_l}$ be a cluster value of $\{\nu^+_{r,m}\}_{m \geq 1}$.  We define
$$T^+_r = \int_{\Gc_r} [\Delta] d \nu^+_{r}(\Delta).$$ 
\end{definition}

By construction the above current satisfies $0 \leq T^+_r  \leq T^+$.  The fact that $\mathbf 1_{(1-(2r)^\gamma) P_{2r}} T^{+}_{r,m} $ is uniformly laminar and  a standard argument using Hurwitz Lemma give  that $T^+_r$ is a uniformly laminar current over  $(1-(2r)^\gamma) P_{2r}$. 

By repeating the above constructions replacing $f$ by $f^{-1}$ we obtain currents $T^{-}_{r,m}$ such that $0 \leq T^{-}_{r,m} \leq d^{-m} [f^{m}(L)] $ and as $m$ tends to infinity along a sub-sequence one obtains a measure $\nu^-_r$ on $\Gc_r$ and a current $T^-_r$ that is uniformly laminar current over  $(1-(2r)^\gamma) P_r$ and $0\leq T^-_r \leq T^-$.

It follows from Proposition \ref{prop:dujardin} that 
\begin{equation} \label{eq:mass-T-T_r}
 \|T^\pm -   T^\pm_r  \|_{(1 - (2r)^\gamma ) P_{2r}} \leq  2 Cr^2.
\end{equation}

Since $T^\pm_r$ are both uniformly laminar over  $(1-(2r)^\gamma) P_r$ one can define a positive measure $\overline \mu_r$ by the formula
\begin{equation*}
\overline \mu_r: =  T^+_r  \dot{\wedge} \,  T^-_r  := \iint_{\Delta,\Delta' \in \Gc_r} [\Delta \cap \Delta'] d \nu^+_r(\Delta) d \nu^-_r(\Delta'),
\end{equation*} 
where,  as usual,  $[\Delta \cap \Delta']$ denotes the sum of Dirac masses counted with multiplicity along the isolated points of  $\Delta \cap \Delta'$ and  $ [\Delta \cap \Delta'] = 0$ if $\Delta \cap \Delta'$ has positive dimension, see \cite{dujardin:intersection}.  Observe that $\overline \mu_r$ is a positive measure supported on $(1-(2r)^\gamma) P_r$ whose mass is bounded by one and $\overline \mu_r \leq \overline \mu$.

\begin{proposition} \label{prop:geometric-intersection0}
Let $r$,  $\gamma$ and $\overline \mu_r$ be as above.  Then,
$$\|  \mu -  \overline \mu_r  \| \leq C r^\gamma.$$
for some constant $C>0$ independent of $r$ and $\gamma$.  In particular,  Properties (P1), (P2) and (P3) hold for $T^\pm_{r,0}$ and $\overline \mu_{r,0}$.
\end{proposition}

\begin{proof}
Both $\overline \mu$ and $\overline \mu_r$ are supported by $P_{2r}$,  so it is enough to bound $\big(\overline \mu -  \overline \mu_r  \big)(P_{2r})$.  We start by fixing a cut-off function $\chi_r$ with support in $\D(0;R)^2$ with the following standard properties: $0 \leq \chi_r \leq 1$,  $\chi_r \equiv 1$ over $(1 - (2r)^\gamma) P_{2r} $,   $\chi_r \equiv 0$ over $(1 - r^\gamma) P_{2r} $ and $\ddc \chi_r  \leq A / r^{2+2\gamma} \omegaFS$ for some constant $A>0$.

Writing $P_{2r} = (1-(2r)^\gamma)P_{2r}  \cup \big( P_{2r}  \setminus ((1-(2r)^\gamma) P_{2r} ) \big)$,  using that  $\overline \mu -  \overline \mu_r \leq \overline \mu$ and applying Lemma \ref{lemma:dujardin-fubini} we get that
\begin{equation} \label{eq:geometric-int-1}
\big( \mu -  T^+_r \dot{\wedge} \,  T^-_r  \big)(P_{2r}) \leq  \langle  \overline \mu -  T^+_r \dot{\wedge} \,  T^-_r,  \chi_r \rangle + 16 r^\gamma
\end{equation}
and a result of Dujardin says that $ T^+_r \dot{\wedge} \,  T^-_r =   T^+_r \wedge \,  T^-_r$. Therefore,  it is enough to estimate the integral $ \langle  \overline \mu -  T^+_r  \wedge \,  T^-_r,  \chi_r \rangle$.

We have
\begin{align}  \label{eq:geometric-int-2}
 \mu -  T^+_r  \wedge \,  T^-_r &=  T^+  \wedge \,  T^- -  T^+_r  \wedge \,  T^-_r = \big( T^+ - T^+_r\big) \wedge T^- + T^+_r \wedge \big( T^- - T^-_r\big) \nonumber \\
&\leq \big( T^+ - T^+_r\big) \wedge T^- + T^+ \wedge \big( T^- - T^-_r\big).
\end{align}

Let us bound $ \langle \big( T^+ - T^+_r\big) \wedge T^-,  \chi_r \rangle$.  The bound on $ \langle   T^+ \wedge \big( T^- - T^-_r\big),  \chi_r \rangle$ will be obtained in a similar fashion.

Let $u^-$ be the canonical quasi potential of $T^-$.  That is $T^- = \omegaFS + \ddc u^-$ and $\langle \overline \mu, u^- \rangle = 0$.  Then, 
\begin{align*}
 \langle \big( T^+ - T^+_r\big) \wedge T^-,  \chi_r \rangle &= \int \chi_r  \big( T^+ - T^+_r\big) \wedge T^-  \\
 &= \int \chi_r  \big( T^+ - T^+_r\big) \wedge \ddc u^- + \int \chi_r  \big( T^+ - T^+_r\big) \wedge \omegaFS.
\end{align*}

From \eqref{eq:mass-T-T_r} and the fact that $\omegaFS$ is comparable with the standard K\"ahler form on $P_{2r}$,  we get that the last integral is bounded by a constant times $r^2$,  which is smaller than $r^\gamma$.   For the remaining integral,  denote by $p_\eta$ the center of the cell $\W_{r,\eta}^2$  and let $c_\eta:= u^-(p_\eta)$.   Recall that  $\chi_r \equiv 0$ near the boundary of the cells and $T^+ - T^+_r$ is closed in  $(1 - (2r)^\gamma)P_{2r}$.  Then,  by Stokes' theorem and  \eqref{eq:mass-T-T_r} again, we get
\begin{align*}
\int \chi_r  & \big( T^+ - T^+_r\big)  \wedge \ddc u^-  = \int u^- \,  \ddc  \chi_r  \big( T^+ - T^+_r\big) \\ 
&= \sum_\eta  \int_{\W_{r,\eta}^2} \big(u^- - c_\eta \big) \,  \ddc  \chi_r  \big( T^+ - T^+_r\big) \\
&\leq \frac{A}{r^{2+2\gamma }} \sum_\eta  \int_{(1-(2r)^\gamma) \W_{r,\eta}^2} \big| u^- - c_\eta\big|  \big( T^+ - T^+_r\big) \wedge \omegaFS
\\
&\leq \frac{A r^{\alpha^-}}{r^{2+2\gamma }} \sum_\eta  \int_{(1-(2r)^\gamma) \W_{r,\eta}^2}    \big( T^+ - T^+_r\big) \wedge \omegaFS \\
&\leq  AC r^{2} r^{\alpha^-} r^{-2- 2\gamma } = AC r^{\alpha^- - 2\gamma } \leq AC r^\gamma,
\end{align*}
where in the last step we have used that $3 \gamma < \alpha^-$.

We have thus proved that $ \langle \big( T^+ - T^+_r\big) \wedge T^-,  \chi_r \rangle \leq C' r^\gamma$ for some constant $C' >0$.   Arguing similarly with the quasi-potential $u^+$ of $T^+$ and using that $3 \gamma < \alpha^+$ we obtain $ \langle   T^+ \wedge \big( T^- - T^-_r\big),  \chi_r \rangle \leq C'r^\gamma$.  Together with \eqref{eq:geometric-int-1} and \eqref{eq:geometric-int-2} this completes the proof.
\end{proof}

\begin{proposition}
There exists currents $T^\pm_r:= T^\pm_{r,0}, \,T^\pm_{r,-p},\, \ldots\, , T^\pm_{r,-\ell p}, \, T^\pm_{r,-n}$ and $T^\pm_{r,p}$  satisfying properties (P1), (P2) and (P3).
\end{proposition}

\begin{proof}
By the construction of $T^\pm_r = T^\pm_{r,0}$ and Proposition \ref{prop:geometric-intersection0},  properties (P1) -- (P3) hold at the level zero.  We now construct  $T^\pm_{r,-p},\, \ldots\, , T^\pm_{r,-\ell p}, \, T^\pm_{r,-n}$ and verify the corresponding properties.  As before we will only detail the construction of $T^+_{r,\bullet}$,  that of $T^-_{r,\bullet}$ being analogous after replacing $f$ by $f^{-1}$.

By definition $T^+_{r,0}$ is obtained from a limit of certain components $d^{-m}[f^{-m} (L)]$ along a subsequence  $(m_l)_{l \geq 1}$.  The construction of the current $T^+_{r,-p}$ will follow the same lines,  but considering instead components of $d^{-m-p}[f^{-m-p} (L)]$.  Keeping the above notation, we consider $S^{i,+}_{r,m,-p}$,  $T^{+}_{r,m,-p}$ and get a uniformly laminar current  $$\mathbf 1_{(1-(2r)^\gamma) P_{2r}} T^{+}_{r,m,-p} = \int_{\Gc_r} [\Delta \cap (1-(2r)^\gamma) P_{2r}] d \nu^+_{r,m,-p}(\Delta)$$ for some positive measure $\nu^+_{r,m,-p}$ on the space $\Gc_r$.  Let $(m_l)_{l \geq 1}$ be the subsequence used to define $T^\pm_r$ (see Definition \ref{def:T^+_r}).   By taking the limit $\nu^+_{r,m,-p}$ along a subsequence of $(m_l)_{l \geq 1}$ one gets a limiting measure on $\Gc_r$ that we denote by $\nu^+_{r,-p}$.  We then define
$$T^+_{r,-p} = \int_{\Gc_r} [\Delta] d \nu^+_{r,-p}(\Delta).$$ 

It is clear that $0\leq T^+_{r,-p} \leq T^+$ (Property (P1)) and that $T^+_{r,-p}$ is uniformly laminar over $(1-(2r)^\gamma) P_{2r}$  (Property (P2)).  We can build $T^-_{r,-p}$ by replacing $f$ by $f^{-1}$,  and the proof that  $\overline \mu_{r,\-p}:= T^+_{r,-p} \dot{\wedge} T^-_{r,-p}$ satisfies Property (P3) is analogous the the one of Proposition \ref{prop:geometric-intersection0}.

The remaining currents are obtained analogously. For $T^+_{r,-j p}$, $j=1,\ldots,\ell$ and $j=-1$,  we start from the components of $\pi_i^{-1}(\overline \W_{2r,\eta_i}) \cap f^{-m -jp}(L)$ and follow the above procedure. We then choose a subsequence of  $m$  that is common to all those used in constructing $T^+_{r,0}, T^+_{r,-p}, \ldots, T^+_{r,-(j-1)p}$, and let $m \to \infty$  along this subsequence.  

The currents $T^-_{r,-j p}$ and the measures $\overline \mu_{r,-j p}$, $j=1,\ldots,\ell$ and $j=-1$ are obtained in the same way, and the proofs of properties (P1), (P2), and (P3) remain unchanged.
\end{proof}

We end this subsection with the proof of the compatibility condition (P4)

\begin{lemma}
Let $j=-1,\ldots,\ell$,  if $\Delta^{\u}_{-jp}$ is a plaque of $T^-_{r,-jp}$ and $\Delta^{\u}_{-jp+p}$ is a plaque of $T^-_{r,-jp+p}$  then $f^p(\Delta^{\u}_{-jp})$  is either disjoint from $\Delta^{\u}_{-jp+p}$  or they meet along a common open set. An analogous property holds for the plaques of $T^-_{r,\-n}$ and $T^-_{r,-\ell p}$ under the action of $f^{p+s}$. In other words, Property (P4) holds.
\end{lemma}

\begin{proof}
We only prove the result for $j=1$, the other cases being analogous. By the definition of the current $T^-_{r,-p}$, a plaque $\Delta^{\u}_{-p}$ subordinate to it is given by a limit of $\Delta^{m_l}_{-p}$ as $l \to \infty$, where $\Delta^{m_l}_{-p}$ is a connected component of $f^{m_l -p}(L)$, and a plaque $\Delta^{\u}_{0}$ of $T^-_{r,0}$ is a limit of $\Delta^{m_l}_{0}$ as $l \to \infty$, where $\Delta^{m_l}_{0}$ is a connected component of $f^{m_l}(L)$. Recall that we considered a common subsequence in the construction of $T^-_{r,-p}$ and $T^-_{r,0}$. 

If $f^p(\Delta^{\u}_{-p})$ and $\Delta^{\u}_{0}$ are disjoint, there is nothing to show.  Assume then that $f^p(\Delta^{\u}_{-p})$ and $\Delta^{\u}_{0}$ have non-empty intersection. Suppose by contradiction that $f^p(\Delta^{\u}_{-p})$ and $\Delta^{\u}_{0}$ intersect at an isolated point. By Hurwitz's lemma, for $l$ large enough, $f^p(\Delta^{m_l}_{-p})$ and $\Delta^{m_l}_{0}$ also intersect at an isolated point. However, both $f^p(\Delta^{m_l}_{-p})$ and $\Delta^{m_l}_{0}$ are open holomorphic discs contained in $f^{m_l}(L)$. Since $f^{m_l}(L)$ is a smooth curve in $\mathbb{C}^2$,  these discs cannot intersect at an isolated point. Therefore, $f^p(\Delta^{\u}_{-p})$ and $\Delta^{\u}_{0}$ intersect along a common relatively open set, and the result follows.
\end{proof}

\subsection{Expansion/Contraction estimates -- Properties  (P5) and (P6)}

In this subsection we provide estimates on the action of $f$ and $f^{-1}$ along the plaques of the currents $T^\pm_{r,\bullet}$,  leading to the proof of Properties (P5) and (P6).

\begin{proposition} \label{prop:P5}
Let $\Delta^{\s}$ be a plaque of $ T^+_{r,-jp}$ for some $j=-1, 0,1,\ldots, \ell$,  which  is a graph over $\W_{2r,\eta_i}$ for $i=1$ or $2$.  Denote by $\widetilde{\Delta^{\s}} \subset \Delta^{\s}$ the corresponding graph over $\W_{r,\eta_i}$.  If $x \in \widetilde{\Delta^{\s}}$ and $v^{\s}_x$ is a unit vector tangent to $\widetilde{\Delta^{\s}}$ at $x$, then $$\|D f^p(x) \cdot v^{\s}_x\|  \leq   d^{-\alpha n/4} . $$
An analogous statement hold for the plaques $\Delta^{\u}$  of $ T^-_{r,-jp}$ by replacing $f$ and $f^{-1}$.  Therefore,  Properties  (P5) and (P6) hold.
\end{proposition}

\begin{proof}
Again,  we prove the result only for $j=1$.  Fix  $\Delta^{\s}$ as in the statement.  Without loss of generality,  we can assume that $\Delta^{\s}$ is a graph over the first coordinate axes.  Let $\widetilde{\Delta^{\s}} \subset \widehat{\Delta^{\s}} \subset \Delta^{\s}$ be as in the proof of Lemma \ref{lemma:slope0}.  By the construction of $T^+_{r,-p}$, we have that  $\Area(f^p(\Delta^{\s})) \leq d^{-p/2}$.  Let $M$ be the modulus of $f^p(\Delta^{\s}) \setminus f^p(\widehat{\Delta^{\s}})$ which is the same as the one of  $\Delta^{\s} \setminus \widehat{\Delta^{\s}}$.  Arguing as in the proof of Lemma \ref{lemma:slope0} we get that $\diam (f^p(\widehat{\Delta^{\s}} ) )  \leq M^{-1/2 } d^{-p/4}$.  Denote $\widetilde \varphi(z):=(z,\varphi(z))$.  The above estimate together with  Cauchy's formula yields
$$\|D(f^p \circ \widetilde \varphi)(z)\| \leq c_0 r^{-1} d^{-p/4}, \quad \text{for} \,\,   z \in \W_{r,\eta_1}$$
for some constant $c_0 > 0$ independent of $r$,  $p$ and $n$.

Since $\widetilde \varphi$ is a biholomorphism there exists a vector $w$ tangent to $ \W_{r,\eta_1}$ at $z$ such that $v^{\s}_x= D\widetilde \varphi(z) \cdot w$.  As $x = \widetilde \varphi(z)$ the chain rule and the above inequality give  
$$\|D f^p(x) \cdot v^{\s}_x\| = \|D f^p(x) ( D\widetilde \varphi(z) \cdot w ) \|  = \|D(f^p \circ \widetilde \varphi)(z) \cdot w \| \leq c_0 r^{-1} d^{-p/4} \|w\|.$$

Now,  $w = D\pi_1(x) \cdot v^{\s}_x$ and $\|D\pi_1(x)\| \leq 1$ because $\pi_1$ is a linear projection,   so $\|w\| \leq \| v^{\s}_x\| \leq 1$.  Since $p = \ceil{6\alpha n}$,  $r = d^{-\alpha n}$,  $\alpha >0$ is small and $n$ is large,   we conclude that $\|D f^p(x) \cdot v^{\s}_x\| \leq c_0 d^{\alpha n - 3\alpha n/2} \leq d^{-\alpha n/4}$,  thus giving the desired result.
\end{proof}

\subsection{The set $\overline \Bc^*$ and the angle between stable and unstable directions -- Properties (P7) and (P8)
}

In this subsection we show that there exist a set $\overline \Bc^*$ of large measure over which the angle  between stable and unstable directions can be controlled. This will prove   Properties (P7) and (P8). Recall that $\overline \mu_r = \overline \mu_{r,0}$.

\begin{lemma} \label{lemma:mass-B*}
Let $\overline \Bc_\bullet:= \supp \overline \mu_{r,\bullet}$ and $$\overline \Bc^*:=   f^{-p}(\overline \Bc_{p}) \cap \overline \Bc_0 \cap f^p(\overline \Bc_{-p}) \cap \cdots \cap f^{\ell p}(\overline \Bc_{-\ell p}) \cap f^n(\overline \Bc_{-n}).$$
Then $\overline \mu_r(\overline \Bc^*) \geq 1 - Cr^\gamma$.
   \end{lemma}
   
\begin{proof}
Using that $\overline \mu_r \leq \mu$ and the fact that $\mu$ is invariant by $f$, we can bound $\overline \mu_r((\overline \Bc^*)^c)$ from above by $\sum_{j=-1}^\ell \mu(\overline \Bc_{-jp}^c)  +  \mu(\overline \Bc_{-n}^c)$. Then, for each term in the sum, we write $\mu = (\mu - \overline \mu_{r,-jp}) + \overline \mu_{r,-jp}$ and similarly for the last one. Using that $\overline \mu_{r,-jp}(\overline \Bc_{-jp}^c) = 0$ and that $\|\mu - \overline \mu_{r,-jp}\| \leq C r^\gamma$ (Property (P3)) we get the result.
\end{proof}

   Observe that, by definition,  all the currents $T^\pm_{r,\bullet}$ are uniformly laminar in a neighborhood of every point $x \in \overline \Bc^*$.  In particular, through each such point, there passes a unique plaque of each of these currents. Recall that the constant $\kappa$ is defined in \eqref{eq:kappa-def}.
    
\begin{proposition}
Let $\overline \Bc^*$ be the above set and fix $x \in \overline \Bc^*$. Let $j=-1,0,1,\ldots,\ell$ and denote by $\Delta^{\u}_{-jp}$ (resp.~ $\Delta^{\s}_{-jp}$) the plaque of $T^-_{r,-jp}$ (resp.~ $T^+_{r,-jp}$)  through $x$.  Then  $\measuredangle (\Tan_x \Delta^{\s}_{-jp},\Tan_x \Delta^{\u}_{-jp}) > d^{-\kappa}r^{6 \kappa}$.  In particular $\measuredangle (\Tan_x \Delta^{\s}_{-jp},\Tan_x \Delta^{\u}_{-jp}) > \rho^{7}$.  In other words, Property (P8) holds.
\end{proposition}

\begin{proof}
We only prove the result for $j=0$, the other cases being analogous. Fix $x \in \overline \Bc^*$.

\smallskip

\textit{Claim.} If $v^{\s} \in \Tan_x \Delta^{\s}_{0}$ is a unit  vector, then  $\|Df^p(x) \cdot v^{\s}\| \leq \frac12$. If $v^{\u} \in \Tan_x \Delta^{\u}_{0}$ is a unit  vector, then  $\|Df^p(x) \cdot v^{\u}\| \geq 2$.  

\smallskip

\textit{Proof of claim.} The first inequality follows directly from the first part of Proposition \ref{prop:P5} and the fact that $n$ is large. For the second inequality, by the definition of $\overline \Bc^*$ there exists $z \in \overline \Bc_{p}$ such that $x=f^{-p}(z)$,  so $f^p(x) = z$.  It follows from (P4) that  $w^{\u} = Df^p(x) \cdot v^{\u}  \in \Tan_z \Delta^{\u}_{p}$, where $\Delta^{\u}_{p}$ is the plaque of $T^-_{r,-p}$ through $z$.  The second part of Proposition \ref{prop:P5} gives that $\|Df^{-p}(z) \cdot w^{\u}\| \leq \frac12 \|w^{\u}\|$. Hence
\begin{align*}
1 &= \|v^{\u}\| =   \|D (f^{-p} \circ f^p) (x) \cdot  v^{\u}\| = \|D (f^{-p})(z)  Df^p (x)  \cdot  v^{\u}\| = \|Df^{-p}(z) \cdot w^{\u} \| \\ 
&\leq \frac12 \|w^{\u}\| = \frac12 \| Df^p(x) \cdot v^{\u}\|,
\end{align*}
so $\|Df^p(x) \cdot v^{\u}\| \geq 2$,  thus proving the claim.  $\square$ \smallskip

By definition, the angle between two complex lines $L$ and $L'$ is the minimum of the angles between two real lines on $L$ and $L'$ respectively,  see \cite[A3.2]{chirka}.  Assume,  by contradiction,  that $\measuredangle (\Tan_x \Delta^{\s}_{0},\Tan_x \Delta^{\u}_{0}) \leq d^{-\kappa}r^{6 \kappa}$. By the definition of $\kappa$  and the chain rule we have that $\|Df^p(x)\| \leq d^{p \kappa}$.  Let $v^{\s} \in \Tan_x \Delta^{\s}_{0}$ and $v^{\u} \in \Tan_x \Delta^{\u}_{0}$ be two unit vectors whose angle is less than $d^{-\kappa} r^{6\kappa}$.  In particular $\|v^{\s}-v^{\u}\| \leq \sqrt 2 d^{-\kappa}r^{6 \kappa}$.  Then,  the above claim and the reverse triangle inequality  give $$\frac32 = 2 - \frac12 \leq \|Df^p(x) \cdot v^{\u}\| - \|Df^p(x) \cdot v^{\s}\| \leq \|Df^p(x) \cdot (v^{\u} - v^{\s})\| \leq \sqrt 2 d^{p \kappa} d^{-\kappa}r^{6 \kappa}.$$
But $p= \ceil{6 \alpha n}$ and $r=d^{-\alpha n}$,  so $\frac32 \leq \sqrt 2 d^{6 \alpha n \kappa} d^{-6 \alpha n \kappa} = \sqrt2$ a contradiction.   This proves the first inequality of the proposition. The second inequality is straightforward, because $r=d^{-\alpha n}$ and $n$ is large,  so $d^{-\kappa}r^{6 \kappa} > r^{7 \kappa} = \rho^7$.
\end{proof}

As a consequence of the above estimate one can show that the stable and unstable discs do not have extra intersection points in a cell of controlled size.   Given a plaque $\Delta$ parametrized by a holomorphic function $\varphi: \W_{2r,\eta_i} \to \C$,  we say that $\Delta$ intersects a given set $E$ before time $T >0$ if there exists $z \in \W_{2r,\eta_i} $ such that $|z - p_{\eta_i}| < T$ and $\varphi(z) \in E$.

\begin{corollary}
Let $\overline \Bc^*$ be the above set and fix $x \in \overline \Bc^*$.  In the notation of the last proposition,  the discs $\Delta^{\u}_{-jp}$ and $\Delta^{\s}_{-jp}$ do not meet before time $\rho^{10}$ except at $x$.
\end{corollary}

The above statement can be restated as follows: after replacing the tiling by the cells $\W^2_{r,\eta}$ with the finer tiling given by $\W^2_{\rho^{10},\eta}$, each cell contains at most one intersection point of the stable and unstable plaques associated with the currents $T^\pm_{r,\bullet}$.

\subsection{Working with a single projection -- Properties (P9) and (P10)}

In this subsection, we show that the use of the two coordinate projections $\pi_1$ and $\pi_2$ can be replaced, up to a controlled loss, by a single projection onto a generic line,  relative to which all plaques of our currents are graphs with controlled slope.  This step is necessary to define suitable graph transform with good quantitative properties,  see Section \ref{sec:graph-transform}.

We begin with an elementary lemma.  Fix $\mathtt p\in \mathbb C^2$ and two linearly independent lines $L_1,L_2\subset \mathbb C^2$. 
The splitting $\mathbb C^2=L_1\oplus L_2$ induces a natural affine coordinate system on $\mathbb C^2$ centered at $\mathtt p$. 
If the center is changed from $\mathtt p$ to another point $\mathtt p'$, the resulting coordinate systems differ by a translation. 
Consequently, if a holomorphic disc $\Delta$ is represented in one such coordinate system as a graph $(z,\varphi(z))$, then in the other coordinate system it is represented as $(t,\psi(t))$, where $\varphi$ and $\psi$ differ by a constant.  In particular, $\Lip(\varphi)=\Lip(\psi)$, so the Lipschitz constant depends only on the splitting $\mathbb C^2=L_1\oplus L_2$, and not on the choice of center. 

\begin{lemma}  \label{lemma:frames}
Fix $0<\rho<1$.  Let $\Delta$ a holomorphic disc in $\C^2$ which is a graph over a square in $\C$ of side $4r$.  Fix $\mathtt{p} \in \Delta$ and let $L$ be a line through $\mathtt{p}$ such that $\measuredangle (\Tan_{\mathtt{p}} \Delta, L) \geq \rho^{10}$.  

\begin{enumerate}
\item   Denote by  $(z,w)$ an affine holomorphic coordinate system induced by the splitting $\C^2 = \Tan_{\mathtt{p}} \Delta \oplus L$.  Then $\Delta$ is a graph of the form $(z,\psi(z))$ over a square of size $\frac{1}{4} \rho^{14}$ with $\Lip(\psi) \leq r$.

\item Let $L$ be as above and denote by  $(t,s)$ an affine holomorphic coordinate system induced by the splitting $\C^2 = L^\perp \oplus L$.  Then $\Delta$ is a graph of the form $(t,\varphi(t))$ over a square of size $\frac{1}{32} \rho^{24}$ with $\Lip(\varphi) \leq \rho^{-10}$.
\end{enumerate}
\end{lemma}

\begin{proposition} \label{prop:one-projection}
There exists a line $D$ in $\C^2$ and positive currents   $\T^\pm_{r,\bullet}$ with the following properties.
\begin{enumerate}
\item $\T^\pm_{r,\bullet} \leq T^\pm_{r,\bullet} \leq T^\pm$
\item $\T^\pm_{r,\bullet}$ are uniformly laminar
\item All plaques of $\T^\pm_{r,\bullet}$ are graphs with respect to the orthogonal projection $\pi_D: \C^2 \to D$ above squares of side $\rho^{25}$ whose slope is bounded from above by $1/\rho^{10}$.
\item $\|T^\pm - \T^\pm_{r,\bullet}\|_{P_{2r}} \leq 100 r^\gamma$ 
\item We have $0 \leq \T^+_{r} \dot{\wedge} \T^-_{r} \leq \mu_r$ and  $\|\mu -  \T^+_{r,\bullet}  \dot{\wedge}  \T^-_{r,\bullet} \| \leq C r^\gamma$

\end{enumerate}
\end{proposition}

The conclusion of the above proposition is that Properties (P9) and (P10) hold.  As already observed in Remark \ref{rmk:inherit} the properties already obtained for $T^\pm_{r,\bullet}$ and inherited by $\T^\pm_{r,\bullet}$.

We begin some notations and a preliminary lemma.  We will detail the proof in  the case of $T^{\pm}_{r}$ since the argument works for the other currents.  Recall that for  $x \in \supp T^+_{r}$ there's a unique plaque $\widetilde{\Delta}^{\u}(x)$ of $ T^+_{r}$ passing through $x$.  Moreover $\widetilde{\Delta}^{\u}(x)$ is  a graph above a square of side $4r$ with respect to one of the two canonical projections.  Similarly for the plaques $\widetilde{\Delta}^{\s}(x)$ of $ T^-_{r}$.
 
 Consider the sets 
 $$E^\pm:= \{(x,L) \in \C^2 \times \P^1 : x \in \supp T^{\pm}_{r}  \, \text{ and }\,   \measuredangle (\Tan_{x}\Delta^{\u/\s}(x), L) < \rho^{10} \}$$ and for fixed $L \in \P^1$ let $E^\pm(L):= \{x \in \C^2 : (x,L) \in E^{\pm}\}$.  Denote by $\sigma^\pm_r:= T^\pm_r \wedge \omegaFS$ the trace measure of $ T^\pm_r $ and set 
 $$B^\pm:= \{L \in \P^1 : \sigma^\pm_r(E^\pm(L)) \geq \rho^{10} \}.$$
 
 Define also 
 $$E:=\{(x,L) \in \C^2 \times \P^1 : x \in \supp \mu^{\pm}_{r},  \,\,    \measuredangle (\Tan_{x}\Delta^{\u}(x), L) < \rho^{10} \,\,  \text{ or } \,\,   \measuredangle (\Tan_{x}\Delta^{\s}(x), L) < \rho^{10} \}$$
 and 
  $$B:= \{L \in \P^1 : \mu_r(E^\pm(L)) \geq \rho^{10} \}.$$
  
 \smallskip

Using the natural parallelism of $\C^2$,  we identify two parallel lines as being the same.  In this way,  the space of all lines in $\C^2$ is identified with $\P^1$.  In what follows, the Lebesgue measure on the space of lines in $\C^2$ is, by definition,  the probability measure induced by the Fubini-Study form on $\P^1$.  We will denote it by $\Leb$.

\begin{lemma} \label{lemma:fubini-P1}
Denote by $\Leb$ the uniform measure on $\P^1$.  There exists a constant $C>0$  such that $\Leb(B^\pm) \leq C \rho^{10}$ and $\Leb(B) \leq C \rho^{10}$.
\end{lemma}

\begin{proof}
We begin by observing that,  for fixed $L$,  the set  $\{L':  \measuredangle ( L' , L) < \rho^{10} \}$ has Lebesgue measure bounded by  $C \rho^{20}$ for some constant $C>0$.  Since $T^\pm_r \leq T^\pm$,  we have that $\sigma^\pm_r \leq T^\pm \wedge \omegaFS$ which is of mass one.  It follows from Fubini's theorem that $\big(\sigma^\pm_r \otimes \Leb \big) (E^\pm) \leq C \rho^{20}$.  By Fubini again,   $\int_{\P^1} \sigma^\pm_r(E^\pm(L)) \,  d\Leb(L) \leq C \rho^{20}$.  In particular,  for every $A>0$,  we have $\Leb (\{L:  \sigma^\pm_r(E^\pm( L )) \geq A \}) \leq A^{-1} C \rho^{20}$.  By taking $A= \rho^{10}$ we obtain the first inequality in the claim. The second inequality is obtained analogously.  
\end{proof}

\begin{proof}[Proof of Proposition \ref{prop:one-projection}]
 We will provide the construction of $\T^\pm_{r,0}:= \T^\pm_{r}$ from $T^\pm_{r,0}= T^\pm_{r}$ and prove its desired properties.  The other cases can be treated analogously,  starting from the other currents $T^\pm_{r,\bullet}$ .  By construction, $T^\pm_{r}$ is formed by graphs with respect to one of the two canoncial projections from $\C^2$ to $\C$.  We'll show that,  by considering a generic line  and by restricting the size of the squares in the base,  we can work with a single projection.  The proof will use Lemma \ref{lemma:frames} and the Fubini-type estimate of Lemma \ref{lemma:fubini-P1}. 
 
We keep the above notation.  Fix a line $L$ in the complement of $B^+ \cup B^- \cup B$.  Such a line exits because,  by  Lemma \ref{lemma:fubini-P1} we have $\Leb(B^+ \cup B^- \cup B) \lesssim \rho^{10} \ll 1 $.  Then,  $\measuredangle (\Tan_{x}\Delta^{\u/\s}(x), L) \geq \rho^{10}$ for all $x$ outside a set of $\sigma^\pm_r$-measure at most $ \rho^{10}$ and $\measuredangle (\Tan_{x}\Delta^{\u}(x), L) \geq \rho^{10}$ and  $ \measuredangle (\Tan_{x}\Delta^{\s}(x), L) \geq \rho^{10}$ for all $x$ outside a set of $\mu_r$-measure at most $ \rho^{10}$.

Let $D$ be a line perpendicular do $L$ and denote $\pi_D:\C^2 \to D$ the orthogonal projection over $D$.  We will work with the coordinate system $(\mathtt p_0,  D \oplus D^\perp)$ where  $\mathtt p_0$ is any fixed reference point.  Observe that $D^\perp$ is parallel to $L$.  In particular,  the above angle estimates hold for $D^\perp$ instead of $L$.

 Let $\Delta := \Delta^{\u}(x)$ be a plaque of $ T^+_{r}$.  By the construction given in Subsection \ref{subsec:P1-P4},  $\Delta$ is a graph above some $\W_{2r,\eta_i}^2$,  for $i=1$ or $2$.  Moreover,  it intersects the smaller cell $ (1-(2r)^\gamma) \overline \W_{2r,\eta}^2$.  If $x \in \Delta \cap  (1-(2r)^\gamma) \overline \W_{2r,\eta}^2 $ and $\measuredangle (\Tan_{x}\Delta, L) \geq \rho^{10}$ then,  by Lemma \ref{lemma:frames},  $\Delta$ is a graph of the form $(t,\varphi(t))$ over a square of side $\frac{1}{32} \rho^{24}$ with $\Lip(\varphi) \leq \rho^{-10}$.  As $\rho^{25} \leq \frac{1}{32} \rho^{24}$,  $\Delta$ is also a graph over a square $Q_{\rho^{25}}$ of $D$ of side $\rho^{25}$.  

Let $\nu^+_{r}$ be the transverse measure associated with $T^+_r$ (see Definition \ref{def:T^+_r}).  Let $\Nc^+_{r}:= \mathbf 1_\Ec \nu^+_{r}$ where $\Ec$ is the set plaques $\Delta$ satisfying $\measuredangle (\Tan_{x}\Delta, L) \geq \rho^{10}$ for all $x \in (1-(2r)^\gamma)P_{2r}$ and put
$$\T^+_r = \int_{\Gc_r} [\Delta] d \Nc^+_{r}(\Delta).$$ 

By an analogous procedure one can define $\T^-_r $.  From the above discussion, these currents satisfy properties (1), (2), and (3) in the statement of the proposition. It remains to prove properties (4) and (5).
 
By construction, the difference $T^+_r - \T^+_r$ is supported by $E^+(L) \cap (1-(2r)^\gamma) P_{2r}$.  Writing 
$T^+ - \T^\pm_{r} = (T^+  - T^+_r ) + (T^+_r - \T^+_{r})$, we get that 
\begin{align*}
\|T^+ - \T^+_{r}\|_{P_{2r}} &\leq T^+ \wedge \omegaFS ( (1-(2r)^\gamma) \mathbb S_r \cap P_{2r} ) \\ & \quad + (T^+ - T^+_r) \wedge \omegaFS ((1-(2r)^\gamma)P_{2r})) + T^+_r \wedge \omegaFS (E^+(L)),
\end{align*}
where we use the notation of Definition \ref{def:manhattan}. Using Lemma \ref{lemma:dujardin-fubini},  the estimate \eqref{eq:mass-T-T_r} and the fact that $L \notin B^+$ yields
$$\|T^+ - \T^+_{r}\|_{P_{2r}} \leq c r^\gamma + 2C r^2 + \rho^{10} \leq 100 r^\gamma,$$
since $\rho < r$ and $r$ is small.  This gives property (4) of the proposition for $\T^+_{r}$ and an analogous argument gives the same property for $\T^-_{r}$.

By construction, we have that  $0 \leq \T^+_{r} \dot{\wedge} \T^-_{r} \leq \overline \mu_r \leq \mu$.  In order to obtain the estimate stated in  property (5),  we argue as above and use Proposition \ref{prop:geometric-intersection0} and the mass estimate for $\mu_r(E^\pm(L))$ coming from the fact that $L \notin B$.  This finishes the proof of the proposition for $\T^{\pm}_{r} $.

The same arguments can be applied for the other (finitely many) currents $\T^{\pm}_{r,\bullet}$.   In principle,  for each current $\T^{\pm}_{r,-jp}$ one finds a line $D_j$ for which the desired properties hold.  However,  since  by the above arguments the set of ``bad'' lines is of measure  $\lesssim \rho^{10} \ll 1$,  one can choose a same line $D$ that works for every current $\T^{\pm}_{r,\bullet}$.  We have thus proved points (1) to (4). The proof of (5) is analogous as the one of Lemma \ref{lemma:mass-B*}. This finishes the proof of the proposition.
\end{proof}

\begin{remark}
For later use, we also note that, in the notation of the previous proof,  from the fact that the diameter of $Q_{\rho^{25}}$ is $\sqrt{2}\rho^{25}$ and that $\Lip(\varphi) \leq \rho^{-10}$,  it follows that the portion of $\Delta$ above $Q_{\rho^{25}}$ is included in the cell $ \overline \W_{2r,\eta}^2 $. 
\end{remark}

\begin{proposition} \label{prop:angle-holder}
 Let $\Bc^*$ be as above and fix $x,y \in \Bc^*$ with $\dist(x,y) < \rho^{40}$ and $\Delta^{\u}(x) \cap \Delta^{\u}(y) \neq \varnothing$.  Then $\measuredangle (\Tan_x \Delta^{\u}(x) , \Tan_y \Delta^{\u}(y)) < \rho^{10}$
\end{proposition}

\begin{proof}
Use Lemma  \ref{lemma:frames} or the transverse Hölder regularity of the unstable lamination,  see Proposition \ref{prop:holo-motion} below.
\end{proof}

\section{Quantitative graph transforms and a closing lemma} \label{sec:graph-transform}

In this section we'll use the construction of the currents $\T^\pm_{r,\bullet}$ in order to define graph transforms with quantitative control.  For the classical theory of graph transforms,  we refer to \cite{katok-hasselblat,barreira-pesin} for real dynamical systems, and to \cite{bedford-lyubich-smillie,dethelin-nguyen} in the holomorphic setting.  As a byproduct,  we will obtain a quantitative closing lemma (see Theorem \ref{thm:graph-transform} below)

\subsection{Holomorphic motions, straightening and stable boxes} \label{sec:motion}

Recall from the last section that the currents  $\T^+_{r,\bullet}$ are formed by graphs above cells of a tilling of given line $D$ by squares of side $\rho^{25}$.  In this subsection,  we'll work with the affine coordinates $(t,s)$ in $\C^2$ induced by the splitting $D \oplus D^\perp$ centered at the intersection point of $D$ and $D^\perp$.  The Julia set of $f$ is contained in a some bi-disc $K := \D(0;R')^2$ in such coordinates.

\medskip

\textbf{Notation:} Fix $\epsilon>0$, possibly depending on $n$ (as will indeed be the case below). Let $\Qc_\epsilon$ (resp. $\Qc_\epsilon^\perp$) denote the tiling of $D$ (resp. $D^\perp$) by squares of side length $\epsilon$, and let $Q_\epsilon$ (resp. $Q_\epsilon^\perp$) denote one of the squares in $\Qc_\epsilon$ (resp. $\Qc_\epsilon^\perp$).  To simplify the notation,  we will often assume that the squares under consideration are centered at the origin, noting that the choice of center is irrelevant to the discussion. Indeed, this amounts only to translating the affine coordinates $(t,s)$ and has no effect on the estimates below.

\medskip

\textbf{$N \times N$ grids.}  It is a standard fact that holomorphic motions induced by stable and unstable leaves are transversally Hölder continuous, with exponent close to one near  central fiber. In fact, this is true in general for any Riemann surface lamination in dimension two, see Proposition~\ref{prop:holo-motion} below. In order to weaken the effect of this Hölder exponent, we will regard each cell inside an $N$ by $N$ grid and restrict our analysis to the central cell of this grid. With this purpose in mind, we set $$\theta_-:= \frac{1-\sqrt2 / N}{1+\sqrt2 / N} \quad \text{ and } \quad \theta_+:= \frac{1+\sqrt2 / N}{1-\sqrt2 / N}$$
and choose $N:=N(f)$ sufficiently large so that  $$\frac12 < \theta_-< 1,  \quad 1< \theta_+ <2 \quad \text{and} \quad 0< 100 \kappa \Big( \frac{\theta_+}{\theta_-} - 1  \Big)< \gamma.$$
Observe that $N$ is independent of $n$ and $\alpha$.

Note that, since $N$ is fixed, $\rho=d^{-\alpha\kappa n}$, and $n$ is sufficiently large, we may assume that $N\rho^{26}\ll \rho^{25}$. In particular, every plaque of $\T^+_{r,\bullet}$ can be represented as a graph over a square $Q_{N\rho^{26}}$ contained in $D$. We will work separately on the strips lying above $Q_{\rho^{26}}$ and $Q_{N\rho^{26}}$. 

The following result is a consequence of the $\lambda$-lemma and the theory of quasiconformal maps.  It can be proven using Harnack's inequality for harmonic functions applied in the discs $\{|t| < \sqrt 2 \rho^{26}\} \subset \{|t| < N \rho^{26}\}$,  see \cite{dujardin-guedj}.  

\begin{proposition} \label{prop:holo-motion}
Let $1/2 <\theta_-<\theta_+<2$ be as above.  Fix a square $Q_{N \rho^{26}}$ of the tilling $\Qc_{N \rho^{26}}$. Assume that $Q_{N \rho^{26}}$ is centered at the origin. Let $(0,a) \in \supp \T^+_{r,\bullet}$ and $\Delta_a = \{(t, \varphi_a(t)) : |t| < N \rho^{26}\}$ be the plaque of $\T^+_{r,\bullet}$ over $Q_{N \rho^{26}}$ passing through $(0,a)$.  Then,
\begin{equation}
\frac12 |a-b|^{\theta_+} \leq |\varphi_a(t) - \varphi_b(t)| \leq 2 |a-b|^{\theta_-} \quad \text{for all } t \in Q_{\rho^{26}}.
\end{equation}
\end{proposition}

Together with Slodkowski extension theorem \cite{slodkowski},  the above statement can be rephrased as follows. Consider a vertical strip $\mathcal T_{\rho^{26}} := \pi_D^{-1}(Q_{\rho^{26}}) \cap K$.   Denote by $\Lc^+_{r}$ the lamination associated with the uniformly laminar current $ \T^+_{r}$ over $\mathcal T_{\rho^{26}}$ .  Then,  there exists straightening homeomorphism $\Phi: \mathcal T_{\rho^{26}} \to \mathcal T_{\rho^{26}}$ of the form $\Phi(z,w) = (z,h_z(w))$ such that both $\Phi$ and $\Phi^{-1}$ are  Hölder with  exponent $\frac12 <\theta_-<1$,  holomorphic in $w$,  that sends $\Lc^+_{r}$ to a trivial lamination of the form $Q_{\rho^{26}} \times \Sigma$ for some closed  subset $\Sigma \subset D^\perp$.  

\smallskip

Consider the refinement $\Qc_{\rho^{60}}$ of the original tiling of $D$ into squares of side length $\rho^{60}$, and the refinement $\Qc^\perp_{\theta_-^{-1}\rho^{60}}$ of the tiling of $D^\perp$ into squares of side length $\theta_-^{-1}\rho^{60}$.

\begin{definition} \label{def:stable-box}
Let $\Phi$ be the above straightening map and denote $\Psi$ the inverse of $\Phi$. 
A vertical strip (of size $\rho^{60}$) is a set of the form $$\mathcal T:=  \mathcal T_{\rho^{60}}:= \pi_D^{-1}(Q_{\rho^{60}}) \cap K$$ for some $Q_{\rho^{60}} \in \Qc_{\rho^{60}}$.  A vertical corridor is a set of the form $\pi_D^{-1}\big(Q_{\rho^{60}} \setminus r^\gamma Q_{\rho^{60}} \big)) \cap K$.

A horizontal strip is a set of the form $\Psi\big( \pi_{D^\perp}^{-1}\big( Q_{ (\rho^{100}/2)^{1 / \theta_-}}\big) \big) \cap \mathcal T$ and  horizontal corridor is a set of the form $$\Psi\big(\mathbf S^\perp \big) \cap \mathcal T,  \quad \text{where} \quad  \mathbf S^\perp:=   \big(  Q^\perp_{ (\rho^{100}/2)^{1 / \theta_-}} \big) \setminus    \big(  r^\gamma Q^\perp_{ (\rho^{100}/2)^{1 / \theta_-}} \big)$$
  for some $Q^\perp_{ (\rho^{100}/2)^{1 / \theta_-}}  \in \Qc^\perp_{ (\rho^{100}/2)^{1 / \theta_-}} $.  

A stable box  (of size $\rho^{60} \times \rho^{100}$)  is a set of the form  $$\Sc = \Psi\big(Q_{\rho^{60}} \times Q^\perp_{ (\rho^{100}/2)^{1 / \theta_-}}  \big)$$  for some $Q_{\rho^{60}} \in \Qc_{\rho^{60}}$ and  $Q^\perp_{ (\rho^{100}/2)^{1 / \theta_-}}  \in \Qc^\perp_{ (\rho^{100}/2)^{1 / \theta_-}} $.
\end{definition}

Observe that,  differently from the vertical ones,  the horizontal strips and corridors are not straight.  The name stable box comes from the fact that the horizontal boundary is aligned with the plaques of $ \T^+_{r}$ .  Observe also that the horizontal strips and corridors are only defined above a given vertical strip $\mathcal T$.    

If we choose another vertical strip $\mathcal T'$ we get another straightening homeomorphism $\Phi'$ with the same properties.  The union of all such strips cover $K$.  It follows that the union of all stable boxes form a partition over $K$.  This partition is a deformed version of the initial partition by the cells $\W_{r,\eta}^2$,  where the horizontal boundaries follow the stable plaques.  Notice that each stable box is saturated by the lamination $\Lc^+_{r}$ in the sense that if a stable box contain a point $x \in \supp \T^+_{r}$,  then it contains the whole plaque of $\T^+_{r}$ through $x$. The horizontal corridors are also saturated in the above sense.

The following result is an analogue of Lemma \ref{lemma:dujardin-fubini} and can be proven similarly. 

\begin{lemma} \label{lemma:dujardin-fubini2}
Denote by $\Cc$ the union of all vertical and horizontal corridors intersecting $K$ and let  $\sigma_{T^\pm}=T^\pm \wedge \omegaFS$ be the trace measure of $T^\pm$.  There exists a constant $C>0$ such that $$\big( \mu + \sigma^+ + \sigma^- \big)(\Cc) \leq C r^\gamma.$$
\end{lemma}

\subsection{Space of graphs and the fixed stable graph}

From now on,  we fix a stable box $\Sc$ of size $\rho^{60} \times \rho^{100}$ as in Definition \ref{def:stable-box}.  We'll define a graph transform on $\Sc$.

\begin{definition}  \label{def:stable-like-graphs}
Fix a stable box $\Sc$ and  $x \in   \Bc^* \cap \Sc$.   Let $\Delta^{\s}$ (resp.  $\Delta^{\u}$) the plaque of $\T^+_{r} $ (resp. $\T^-_{r}$) through $x$.  Denote by $(z,w)$ the affine coordinates associated with the splitting $\Tan_x \Delta^{\s} \oplus \Tan_x \Delta^{\u} $ centered at $x$.
We will denote by 
 $$\Gc_\Sc:= \Big \lbrace \Gamma : \Gamma  \text{ is a graph of the form } (z,\varphi(z)) \text{ over } |z| \leq 2\rho^{60} \text{ s.t. } \Lip(\varphi) \leq 3/4,  |\varphi(0)| \leq \rho^{60} \Big \rbrace$$
and call it the space of stable-like graphs inside $\Sc$.
\end{definition}
  By the Arzelà–Ascoli theorem, this is a compact metric space with the distance
$$d(\Gamma, \Upsilon) = \sup_{ |z| \leq 2\rho^{60}} \big| \varphi(z) - \psi(z) \big|,$$
where $\Gamma$ is the graph of $\varphi$ and $\Upsilon$ that of $\psi$.

The goal of this section is to prove the following key result.  The final conclusion has the form of a closing lemma: every sufficiently good near return is accompanied by a nearby periodic orbit.

\begin{theorem}[Graph transform and Closing Lemma] \label{thm:graph-transform} Fix $x \in   \Bc^* \cap \Sc$ and assume that $\dist(x,f^{-n}(x)) < \rho^{70}$. Then, the map $f^{-n}$ induces a map $\Lambda: \Gc_\Sc \to \Gc_\Sc$ with the following properties.
\begin{enumerate}
\item  For every $\Gamma, \Upsilon \in \Gc_\Sc$ we have that $d(\Gamma,   \Upsilon) \leq  \frac{1}{100} d(\Gamma,\Upsilon) $. In particular, $\Lambda$ admits a unique fixed point in $\Gc_\Sc$, denoted by $\Gamma(x)$. 

\item $f^n$ preserves $\Gamma(x)$ and the induced map $f^{n}: \Gamma(x) \to \Gamma(x)$ is a contraction with respect to the euclidean metric.  More precisely,  for every $y \in \Gamma(x)$ and $u \in \Tan_y \Gamma(x)$ of unit norm, we have that $$\|Df^n(y) \cdot u\| \leq d^{-n/100}.$$  In particular $\Gamma(x)$ contains a periodic point of order $n$ for $f$.
\end{enumerate}

\end{theorem}

The proof is rather long and require some notation and a number of preparatory steps.  The notation and some of the estimates will follow closely those in \cite{dethelin-nguyen}.

The following lemma follows easily from the construction of $\T^+_{r} $ and the proof is left to the reader.
\begin{lemma} \label{lemma:plaques-in-G}
Let $\Delta$ be plaque of $\T^+_{r} $ inside $\Sc$.  Then $\Delta \subset \Gamma$ for some $\Gamma \in \Gc_\Sc$.
\end{lemma}

\subsection*{Pesin frames}

Fix $x \in   \Bc^* \cap \Sc$ as above.  By the definition of $\Bc^*$,  $x \in \Bc_0 \cap f^p(\Bc_{-p})$.  Denote by $\Delta^{\u}(x)$ (resp.   $\Delta^{\s}(f^{-p}(x))$) the plaque of $ \T^-_{r}$ through $x$ (resp.  of $ \T^+_{r,-p}$ through $f^{-p}(x)$).  Properties (P5) and (P6) give  the following:
\begin{align*}
&\bullet \, \text{For a unit vector } v_0 \in \Tan_x  \Delta^{\u}(x) \text{ we have } |Df^{-p}(x)\cdot v_0| \leq d^{-\alpha n},  \text{ and } \\ 
&\bullet \, \text{For a unit vector } \widetilde{u}_0 \in \Tan_{f^{-p}(x)}  \Delta^{\s}(f^{-p}(x)) \text{ we have } |Df^{p}(f^{-p}(x))\cdot \widetilde{u}_0| \leq d^{-\alpha n}.
\end{align*}

For $l \geq 0$, set  
$$u_0:=  \frac{Df^{p}(f^{-p}(x))\cdot \widetilde{u}_0}{|Df^{p}(f^{-p}(x))\cdot \widetilde{u}_0|}, \quad u_{-l} := \frac{Df^{-l}(x) \cdot u_0}{|Df^{-l}(x) \cdot u_0|} \quad \text{and} \quad v_{-l}:= \frac{Df^{-l}(x) \cdot v_0}{|Df^{-l}(x) \cdot v_0|}.$$

Consider also the complex lines
$$E^{\s}(f^{-l}(x)):= \vect(u_{-l})  \quad \text{and} \quad E^{\u}(f^{-l}(x)):= \vect(v_{-l}).$$

We'll denote by $C(f^{-l}(x))$ the complex $2 \times 2 $ matrix sending the canonical basis of $\C^2$ to the basis $\{u_{-l}, v_{-l}\}$. Recall that $p= \ceil{6 \alpha n}$ and $n= (\ell +1) p  + s$.

\smallskip

Property (P8) yields the following estimate on the the above matrices.  The proof is standard.

\begin{lemma} \label{lemma:pesin-matrix}
Let $l \in  \{0, p, \ldots,\ell p,  n\}$.  Then $\big\| \big(C(f^{-l}(x)\big)^{\pm1} \big\| \leq 2 \rho^{-6}$.
\end{lemma}

The following lemma is an easy consequence of the compatibility of the plaques under the action of $f^p$ (Property (P4)) together with the contraction estimates given by properties (P5) and (P6). We leave the details to the reader.

\begin{lemma} \label{lemma:u-v-contraction}
For $j=0, \ldots,\ell - 1$,  we have
$$|Df^{-p}(f^{-jp}(x))\cdot v_{-jp}| \leq d^{-\alpha n/4} \quad \text{and} \quad |Df^{-p}(f^{-jp}(x))\cdot u_{-jp}| \geq d^{-\alpha n/4},$$
and 
$$|Df^{-p-s}(f^{-\ell p}(x))\cdot v_{-\ell p}| \leq d^{-\alpha n/4} \quad \text{and} \quad |Df^{-p-s}(f^{-\ell p}(x))\cdot u_{-jp}| \geq d^{-\alpha n/4}.$$
\end{lemma}

\subsection*{Local form in Pesin coordinates}

As before, we fix $x \in   \Bc^* \cap \Sc$.  Let $\tau_y: \C^2 \to \C^2$ be the translation by $y \in \C^2$.  For  $j=0, \ldots,\ell - 1$,  we set
\begin{equation} \label{eq:def-g_j}
g_j: = C(f^{-(j+1)p}(x))^{-1} \circ \tau^{-1}_{f^{-(j+1)p}(x)} \circ f^{-p} \circ \tau_{f^{-jp}(x)} \circ C(f^{-jp}(x))
\end{equation}
and for the last step
$$g_n: = C(f^{-n}(x))^{-1} \circ \tau^{-1}_{f^{-n}(x)} \circ f^{-p-s} \circ \tau_{f^{-\ell p}(x)} \circ C(f^{-\ell p}(x)),$$
seen as maps defined in a neighborhood of the  origin in $\C^2$ equipped with the canonical frame.

Notice that $g(0)=0$. Moreover, the differential $Df^{-jp}(x)$ maps ${E^{\s}(f^{-jp}(x)),E^{\u}(f^{-jp}(x))}$ to ${E^{\s}(f^{-(j+1)p}(x)),E^{\u}(f^{-(j+1)p}(x))}$, and $Dg_j(0)$ is precisely its representation in the canonical frame $\{e_1,e_2\}$ centered at the origin:
\begin{align*}
\{e_1,e_2\}
&\xrightarrow{\;\tau_{f^{-jp}(x)}\circ C(f^{-jp}(x))\;}
\{E^{\s}(f^{-jp}(x)),E^{\u}(f^{-jp}(x))\} \\
&\xrightarrow{\;Df^{-jp}(x)\;}
\{E^{\s}(f^{-(j+1)p}(x)),E^{\u}(f^{-(j+1)p}(x))\} \\
&\xrightarrow{\;C(f^{-(j+1)p}(x))^{-1}
\circ \tau^{-1}_{f^{-(j+1)p}(x)}\;}
\{e_1,e_2\}.
\end{align*}

In order to keep track of the different frames involved in the construction we introduce the following notation.

\begin{definition} \label{def:frames}
 We denote by  $\Cc$  the canonical frame $\{e_1,e_2\}$ centered at the origin and by $\Ec_{-jp}$ the  frame $\{E^{\s}(f^{-jp}(x)),E^{\u}(f^{-jp}(x))\}$ centered at $f^{-jp}(x)$ (and similarly for $\Ec_n$). 
\end{definition}

The above maps are related by the following diagram

\begin{center}
\begin{tikzpicture}[>=stealth, node distance=2.5cm]

  \node (A) {$\Ec_0$};
  \node (B) [right of=A] {$\Ec_{-p}$};
  \node (C) [right of=B] {$\Ec_{-2p}$};
  \node (D) [right of=C] {$\cdots$};
  \node (E) [right of=D] {$\Ec_{-\ell p}$};
  \node (F) [right of=E] {$\Ec_{-n}$};

  \node (A2) [below of=A] {$\Cc$};
  \node (B2) [below of=B] {$\Cc$};
  \node (C2) [below of=C] {$\Cc$};
  \node (D2) [below of=D] {$\cdots$};
  \node (E2) [below of=E] {$\Cc$};
  \node (F2) [below of=F] {$\Cc$};

  \draw[->] (A) -- node[above] {$f^{-p}$} (B);
  \draw[->] (B) -- node[above] {$f^{-p}$} (C);
  \draw[->] (C) -- (D);
  \draw[->] (D) -- node[above] {$f^{-p}$} (E);
  \draw[->] (E) -- node[above] {$f^{-p-s}$} (F);

  \draw[->] (A2) -- node[below] {$g_0$} (B2);
  \draw[->] (B2) -- node[below] {$g_1$} (C2);
  \draw[->] (C2) -- (D2);
  \draw[->] (D2) -- node[below] {$g_{\ell-1}$} (E2);
  \draw[->] (E2) -- node[below] {$g_n$} (F2);

  \draw[->] (A2) -- (A);
  \draw[->] (B2) -- (B);
  \draw[->] (C2) -- (C);
  \draw[->] (E2) -- (E);
  \draw[->] (F2) -- (F);
\end{tikzpicture}
\end{center}
where the vertical arrows denote the composition of the linear maps $C(\cdot)$ with the suitable translations (cf. \eqref{eq:def-g_j}).

We should see the composition $g_n \circ g_{\ell -1} \circ \cdots g_1 \circ g_0$ as a local representation of $f^{-n}$ around the origin of $\C^2$ after going through the linear coordinate changes given by the above translations and Pesin matrices.

\begin{lemma} \label{lemma:g-estimates}
Let $g$ be one the maps $g_j$ or $g_n$ above.  Then,  $g(0) = 0$ and the following estimates hold
\begin{enumerate}
\item $Dg(0) = \begin{pmatrix}
A & 0 \\ 0 & B
\end{pmatrix}$, with $|A| \geq d^{\alpha n /4}$ and $|B| \leq d^{-\alpha n /4}$;
\item $\|D^2 g (x)\| \leq \rho^{-33}$ when  $\|w\| \leq \rho^{14}$.
\item One can write $g(w) = Dg(0) \cdot w + h(w)$ for some holomorphic map $h$ on $\|w\| < 2 \rho^{14}$ such that $$\|Dh(w)\| \leq \frac{\|w\|}{\rho^{33}} \quad \text{when} \quad  \|w\| \leq \rho^{14}.$$
\end{enumerate}
\end{lemma}

\begin{proof}
Given the estimates from Lemmas \ref{lemma:pesin-matrix} and  \ref{lemma:u-v-contraction},  the proof follows  arguments parallel to those in  \cite[Proposition 1.1]{dethelin-nguyen}.  We leave the details to the reader.
\end{proof}

We shall repeatedly use the following result of the first named author and Nguyen.  The original result is valid in any dimension, but we state it for $\C^2$ with the standard euclidean coordinates for simplicity.

\begin{theorem} \cite[Théorème 3]{dethelin-nguyen} \label{thm:DN}
Let $g(X,Y) = (AX+R(X,Y), BY + U(X,Y))$ be a holomorphic map between open balls $\B(0;R_0)$ and  $\B(0;R_1)$ in $\C^2$ with $R_0 \leq R_1$ such that $g(0) = 0$ and $\max\{\|DR(Z)\|, \|DU(Z)\|\} \leq \delta$ over  $\B(0;R_0)$.  Suppose $A \neq 0$ and $|B| < |A|$.  Denote by $\xi = 1 - |B| |A^{-1}| \in ]0,1]$.

Let ${(X,\varphi(X)),,X\in D}$ be a graph contained in $\B(0,R_0)$, over a subset $D$ of $E_1:=\{Y=0\}$, satisfying $\Lip(\varphi)\leq \gamma_0\leq 1$. Assume that $\delta |A^{-1}|(1+\gamma_0)<1$. Then the image of this graph by $g$ is a graph over $\pi_0\bigl(g(\operatorname{graph}\varphi)\bigr)$, where $\pi_0$ denotes the projection onto the $E_1$-coordinates. 

More precisely, if $(X,\psi(X))$ denotes the resulting graph, then $$|\psi(X_1)-\psi(X_2)|\leq\frac{|B|\gamma_0+\delta(1+\gamma_0)}{|A|-\delta(1+\gamma_0)}|X_1-X_2|.$$
 In particular, this Lipschitz constant is at most $\gamma_0$ provided $\delta\leq\varepsilon(\gamma_0,\xi)$. 
 
 Finally, suppose that $\D(0,\alpha)\subset D$ and $|\varphi(0)|\leq\beta$. Then $\pi_0\bigl(g(\operatorname{graph}\varphi)\bigr)$ contains $ \D \left(0,\big(|A|-\delta(1+\gamma_0)\big)\alpha-\delta\beta\right)$, and $$|\psi(0)|\leq(1+\gamma_0)\big(|B|\beta+\delta\beta+\|D^2g\|_{\B(0,R_0)}\beta^2\big),$$ provided  $\delta\leq\varepsilon(\gamma_0,\xi)$.
\end{theorem}

We now give the construction of the graph transform $\Lambda: \Gc_\Sc \to \Gc_\Sc$.   Recall that $x \in   \Bc^* \cap \Sc$ is fixed and $(z,w)$ are the  affine coordinates centered at $x$ with respect to the frame $\Tan_x \Delta^{\s} \oplus \Tan_x \Delta^{\u} $.   Fix an element $\Gamma = \{( z,\varphi(z) )\} \in \Gc_\Sc$.  In the  notation introduced above,  $\Gamma$ is a holomorphic curve parametrized by
$$ z \longmapsto x + z u_0 + \varphi(z) v_0 \quad \text{for} \quad |z| \leq 2\rho^{60}.$$

Recall that $$g_0 = C(f^{-p}(x))^{-1} \circ \tau^{-1}_{f^{-p}(x)} \circ f^{-p} \circ \tau_{x} \circ C(x).$$  Let $\widetilde \Gamma$ be the image of $\Gamma$ under $C(x)^{-1} \circ \tau_x^{-1}$.  Then $\Gamma$ lives in $\C^2$ with the canonical frame with coordinates $(X,Y)$ and is parametrized by
$$ X \longmapsto  X e_1 + \varphi(X) e_2 \quad \text{for} \quad |X| \leq 2\rho^{60}.$$

\begin{lemma} \label{lemma:g-Gamma}
Fix $\Gamma \in \Gc_\Sc$ and let  $g_0$ and $\widetilde \Gamma$  be as above.  Then $g_0(\widetilde \Gamma)$ is a graph in the canonical frame of $\C^2$ of the form $\widetilde \Gamma_{p} = \{(X,\varphi_{-p}(X))\}$ where $\varphi_{-p}$ is defined on a disc containing $|X| < 20 \rho^{60}$. Moreover $\Lip \varphi_{-p} \leq 1 / 2$ and $|\varphi_{-p}(0)| \leq \rho^{60} / 10$.
\end{lemma}

\begin{proof}
We will use Theorem \ref{thm:DN} and keep its notation.  From Lemma \ref{lemma:g-estimates},  one can write $$g(X,Y) = (AX+R(X,Y), BY + U(X,Y),$$
in the canonical frame of $\C^2$,  where $|A| \geq d^{\alpha n /4}$,  $|B| \leq d^{-\alpha n /4}$ and $R(0,0) = U(0,0) = 0$.  By Lemma \ref{lemma:pesin-matrix},  the graph $\Gamma_0$ is contained in the ball $\B(0;R_0)$ in $\C^2$ with $R_0= \rho^{35}$ and,  again by Lemma \ref{lemma:g-estimates},  we have $\|DR(Z)\| \leq \rho^2$ for all $Z \in \B(0;R_0)$ and similarly for $\|DU(Z)\|$.

We now apply Theorem \ref{thm:DN} with $\delta = \rho^2 $,  $\gamma_0 = 1$,  $\alpha = 2\rho^{60}$ and $\beta = \rho^{60}$.  Recall  that $d^{-\kappa \alpha n}$. Observe first that $\delta |A^{-1}|(1+\gamma_0) \leq 2 \rho^2 d^{-\alpha n / 4} \leq 1$ since $n$ is large.  By Theorem \ref{thm:DN},  $g_0(\Gamma_0)$ is a graph of the form $(X,\varphi_{-p}(X))$ over a disc of radius at least
$$\big(|A|-\delta(1+\gamma_0)\big)\alpha-\delta\beta \geq 2 \big(d^{\alpha n /4} - \rho^2\big) \rho^{60} - \rho^{62} \geq 10 \rho^{60}$$
with 
$$|\varphi_{-p}(0)| \leq (1+\gamma_0)\big(|B|+\delta+\|D^2g\|_{\B(0,R_0)}\beta\big) \beta \leq 2 \big( d^{-\alpha n /4}  + \rho^2 +   \rho^{20} \big) \rho^{60} \leq \rho^{60} / 10.$$

Moreover,  the Lipschitz constant of $\varphi_{-p}$ is bounded by 
$$\frac{|B| + 2\delta}{|A|-2\delta} \leq \frac{d^{-\alpha n /4} + \rho^2}{d^{\alpha n /4} - \rho^2} \leq \frac12,$$
since $n$ is large,  thus completing the proof of the lemma.
\end{proof}

Observe that $\widetilde \Gamma_{p}$ is defined on the larger disc $|X|<20\rho^{60}$. We may therefore restrict it to the original domain $|X|<2\rho^{60}$ and apply $g_1$. Applying the above lemma to $g_1$ instead of $g_0$ yields a new graph $\widetilde \Gamma_{2p} = \{(X,\varphi_{-2p}(X))\}$, where $\varphi_{-2p}$ is defined on $|X|<20\rho^{60}$, satisfies $\Lip \varphi_{-2p}\leq 1/2$, and $|\varphi_{-2p}(0)| \leq \rho^{60}/10$.  Repeating this procedure along the whole composition $g_n \circ g_{\ell -1} \circ \cdots g_1 \circ g_0$ yields a graph $$\widetilde \Gamma_n :=\{(X,\varphi_n(X))\},$$ where $\varphi_n$ is defined on $|X|<20\rho^{60}$, satisfying $\Lip ( \varphi_n ) \leq 1/2$, and $|\varphi_n(0)| \leq \rho^{60}/10$.

With respect to the frames in Definition \ref{def:frames} (see also the diagram after the definition),  the initial graph $\Gamma$ lives in $\Ec_0$ while the graphs $\widetilde \Gamma_0, \widetilde \Gamma_{p}, \ldots, \widetilde \Gamma_{n}$ all live in $\Cc$.

From the construction above and the preceding remark, we obtain a graph $\widetilde \Gamma_n$ in $\Cc$, which can naturally be transported to $\Ec_{-n}$ via the last vertical arrow above.  However, in order to iterate the construction, we would like to view this graph in the initial frame $\Ec_0$ and show that it keeps the initial properties of $\Gamma$.  The following lemma shows that this can be done provided that $f^{-n}(x)$ is sufficiently close to $x$.

\begin{proposition} \label{prop:graph-transform}
Let $x \in   \Bc^* \cap \Sc$.  Fix $\Gamma \in \Gc_\Sc$ and let $\widetilde \Gamma_n  =\{(X,\varphi_n(X))\}$ be the graph obtained be the above procedure.   Let $\Gamma_n$ be the image of $\widetilde \Gamma_n$ under $\tau_{f^{-n}(x)} \circ C(f^{-n}(x))$  in the frame $\Ec_{-n}$. Denote as before by $(z,w)$ the coordinates relative to $\Ec_0$.  Assume that $\dist(x,f^{-n}(x)) < \rho^{70}$.  Then, $\Gamma_n$ is a graph of the form $(z,\Phi(z))$  over the disc $|z| \leq 2\rho^{60}$ with Lipschitz constant bounded by $3/4$ and with $|\Phi(0)| \leq \rho^{60}$.  In other words,  $\Gamma_n \in \Gc_\Sc$.
\end{proposition}
 
 \begin{proof}
 We'll first work within the frame $\Cc$ via the two coordinate changes $\tau_{f^{-n}(x)} \circ C(f^{-n}(x))$ and $C^{-1}(x) \circ \tau_x^{-1}$. In the last step we'll apply $\tau_x \circ C(x)$ to obtain a graph in the frame $\Ec_0$.
 
 We know that in the coordinates $(X,Y)$,  $\varphi_n$ is defined over $|X|<20\rho^{60}$,  $\Lip ( \varphi_n ) \leq 1/2$, and $|\varphi_n(0)| \leq \rho^{60}/10$.  We first consider $$\overline{\Gamma}_n:= \Crm(\widetilde \Gamma_n), \quad \text{where} \quad \Crm:= C^{-1}(x) \circ \tau_x^{-1} \circ \tau_{f^{-n}(x)} \circ C(f^{-n}(x)),$$
 so that $\Gamma_n = \big( \tau_x \circ C(x) \big) (\overline{\Gamma}_n)$.

By setting $D = f^{-n}(x) - x$ have that $ \tau_x^{-1} \circ \tau_{f^{-n}(x)}(Z) = Z + D$ and we can write
\begin{equation} \label{eq:C-decomposition}
\Crm(X,Y) = g(Z) + C(x)^{-1} D, \quad \text{where} \quad g(Z) = C(x)^{-1} C(f^{-n}(x))(Z).
\end{equation}

We'll analyse each term in the above composition separately.

\smallskip

\noindent \underline{\textit{Claim:}} Let $g$ be as above. The image $g(\widetilde \Gamma_n)$ is a graph of the form $(z,\varphi_1(z))$  over the disc $|z| \leq 10\rho^{60}$ with Lipschitz constant bounded by $3/4$ and with $|\varphi_1(0)| \leq \rho^{60} / 5$.

\smallskip

\noindent \textit{Proof of claim:}  Denote by  $y:=f^{-n}(x)$. We have that $\dist(x,y) < \rho^{70}$ from our assumptions. Then,  by Proposition \ref{prop:angle-holder}, we have that $\measuredangle (\Tan_x \Delta^{\u}(x) , \Tan_y \Delta^{\u}(y)) < \rho^{10}$.  Recall that $u_{-l}$ (resp. \ $v_{-l}$) is tangent to $\Delta^{\s}(f^{-l}(x))$ (resp. \ $\Delta^{\u}(f^{-l}(x))$) for  $l \in  \{0, p, \ldots,\ell p,  n\}$.  In particular,  up to multiplying $u_0$ and $v_0$ by a complex number of modulus one,  we can assume that the angle $\theta$ between $u_0$ and $u_{-n}$ and the angle  between $v_0$ and $v_{-n}$ are both smaller than $\rho^{10}$.  In particular,  since the vectors are of unit norm, we have $\|u_0 - u_{-n}\|^2 = 2 - 2 \cos \theta \leq \theta^2$,  so $\|u_0 - u_{-n}\| \leq \rho^{10}$ and similarly for $v_0$ and $v_{-n}$.

We have that $$g(e_1) = C^{-1}(x) u_{-n} = C^{-1}(x) u_{0} + C^{-1}(x) (u_{-n} - u_0) = e_1 + C^{-1}(x) (u_{-n} - u_0)$$ and similarly
$$g(e_2) = e_2 + C^{-1}(x) (v_{-n} - v_0)$$
and we know from Lemma \ref{lemma:pesin-matrix} that $\| C^{-1}(x)\| \leq 2 \rho^{-6}$.  Together with the above estimates for $\|u_0 - u_{-n}\|$ and $\|v_0 - v_{-n}\|$,  the matrix $M_g$ of $g$ in the canonical basis is of the form $$M_g = \begin{pmatrix} 1 + \varepsilon_1 & \varepsilon_2 \\ \varepsilon_3 & 1+ \varepsilon_4 \end{pmatrix} \quad \text{with} \quad |\varepsilon_i| \leq 2 \rho^4, \,  i=1,2,3,4,$$
so $g(X,Y) = (AX  + R(X,Y), BY + U(X,Y))$ with $A= 1 + \varepsilon_1$,  $R(X,Y) = \varepsilon_3 Y$,   $B= 1 + \varepsilon_4$ and $U(X,Y) = \varepsilon_2 X$.  Applying Theorem \ref{thm:DN},  directly gives all the estimates of the claim.  We leave this simple verification to the reader. \hfill $\square$

\smallskip

From the above decomposition of $\Crm$,  we have that $\overline{\Gamma}_n= \Crm(\widetilde \Gamma_n)$ is the image of $g(\widetilde \Gamma_n)$ under the translation by $C(x)^{-1} D =: W = (W_1,W_2)$.  This translation doesn't change the Lipschitz constant and  $\|W \| \leq \|C(x)^{-1}\| \, \|D\| \leq \rho^{60} / 5$,  where we have used Lemma \ref{lemma:pesin-matrix} and the assumption that $\dist(x,f^{-n}(x)) < \rho^{70}$.  Together with the above claim,  we get that $\overline \Gamma_n$ is a graph $(z,\varphi_2(z))$ over $|z| \leq 2 \rho^{60}$ with $\Lip(\varphi_2) \leq \frac34$.  Moreover,  $\varphi_2(0) = \varphi_1(-W_1) + W_2$,  so $$|\varphi_2(0)| \leq |\varphi_1(-W_1) + W_2| \leq  |\varphi_1(-W_1) - \varphi_1(0)| + |\varphi_1(0)| + |W_2| \leq 3 \cdot \rho^{60} / 5 < \rho^{60},$$ where we have used that $\Lip(\varphi_1) \leq 1$ and the fact that  both $|\varphi_1(0)|$ and $\|W\|$ are bounded by $\rho^{60} / 5$.

We thus proved all the properties of the proposition for $\overline \Gamma_n$ in the frame $\Cc$.  In other words, $\overline{\Gamma}_n$ is parametrized by
$$ X \longmapsto  X e_1 + \Phi(X) e_2 \quad \text{for} \quad |X| \leq 2\rho^{60},$$
where $\Lip(\Phi) \leq 3/4$ and $|\Phi(0)| \leq \rho^{60}$ in the frame $\Cc$. Now,  $\Gamma_n = \big( \tau_x \circ C(x) \big) (\overline{\Gamma}_n)$ and $ \tau_x \circ C(x)$ sends $\Cc$ to $\Ec_0$, so in the frame $\Ec_0$ with coordinates $(z,w)$, the graph $\Gamma_n$ is parametrized by
$$ z \longmapsto x + z u_0 + \Phi(z) v_0 \quad \text{for} \quad |z| \leq 2\rho^{60}.$$
This finishes the proof of the proposition.
 \end{proof}

We can now prove the graph transform theorem.

\begin{proof}[Proof of Theorem \ref{thm:graph-transform}-(1)]
Recall that $x \in   \Bc^* \cap \Sc$ is fixed and the space of graphs $\Gc_\Sc$ is defined in terms of the stable/unstable frame $\Ec_0$ at $x$  (see Definition \ref{def:frames} for the notation),  whose associated coordinates are denoted by $(z,w)$. Fix $\Gamma \in \Gc_\Sc$ and let $\Gamma_n \in \Gc_\Sc$ be the graph obtained in Proposition \ref{prop:graph-transform}. We define $\Lambda: \Gc_\Sc \to \Gc_\Sc$ by  $\Lambda(\Gamma):= \Gamma_n$.

Fix $\Gamma = \{ (z,\varphi(z)) \}$ and  $\Upsilon = \{ (z,\psi(z)) \}$ two graphs in $\Gc_\Sc$. Let  $\widetilde \Gamma$ (resp. $\widetilde{\Upsilon}$) be the image of $\Gamma$ (resp. $\Upsilon$) via  $C^{-1}(x) \circ \tau_x^{-1}$ in the canonical frame $\Cc$ with coordinates $(X,Y)$.  As in the end of the proof of Proposition \ref{prop:graph-transform}, we see that this graphs are given by $\{ (X,\varphi(X)) \}$ and  $\{ (X,\psi(X)) \}$ respectively.

By  definition, the image of $\Gamma$ under $\Lambda$ corresponds, in the frame $\Cc$,  to the image of  $\widetilde \Gamma$ under the  composition $g_n \circ g_{\ell -1} \circ \cdots g_1 \circ g_0$, where the maps $g_j$ are defined in \eqref{eq:def-g_j}. We'll use the notations and results of Theorem \ref{thm:DN} and Lemmas \ref{lemma:g-estimates} and \ref{lemma:g-Gamma}. Let $\widetilde \Gamma_{p} = \{(X,\varphi_{p}(X))\}$   and   $\widetilde{\Upsilon}_{p} = \{(X, \psi_{p}(X))\}$ ) be the images of $\widetilde \Gamma$  and $\widetilde{\Upsilon}$ under $g_0$ respectively. 

We have that
\begin{align*}
&g_0(X,\varphi (X)) = \big( AX + R(X,\varphi (X)), B \varphi (X) + U(X,\varphi (X)) \big) =: (X_{-p},Y_{-p})  \\
&g_0(X,\psi (X)) = \big( AX + R(X,\psi (X)), B \psi (X) + U(X,\psi (X)) \big).
\end{align*}

Notice that $\varphi_{-p}(AX + R(X,\varphi (X)) = B \varphi (X) + U(X,\varphi (X)) $ because $(X_{-p},Y_{-p}) \in \widetilde \Gamma_{-p}$. Analogously $ \psi_{-p}(AX + R(X, \psi (X)) = B \psi (X) + U(X, \psi (X)) $.

Introduce
\begin{alignat*}{2}
a_\varphi &\coloneqq AX+R\bigl(X,\varphi(X)\bigr), &\qquad a_\psi &\coloneqq AX+R\bigl(X,\psi(X)\bigr), \\
b_\varphi &\coloneqq B\varphi(X)-U\bigl(X,\varphi(X)\bigr), &\qquad b_\psi &\coloneqq B\psi(X)-U\bigl(X,\psi(X)\bigr).
\end{alignat*}
Using  Theorem \ref{thm:DN} and Lemmas \ref{lemma:g-estimates} and \ref{lemma:g-Gamma}, we get
\begin{align}
\left|\varphi_{-p}(X_{-p})-\psi_{-p}(X_{-p})\right| &= \left|\varphi_{-p}(a_\varphi)-\psi_{-p}(a_\varphi)\right| \nonumber \\
&\leq \left|\varphi_{-p}(a_\varphi)-\psi_{-p}(a_\psi)\right|+\left|\psi_{-p}(a_\psi)-\psi_{-p}(a_\varphi)\right| \nonumber  \\
&\leq \left|b_\varphi-b_\psi\right|+\left|R\bigl(X,\varphi(X)\bigr)-R\bigl(X,\psi(X)\bigr)\right|  \nonumber  \\
&\leq \left(|B|+\lVert DU\rVert_{L^\infty(\mathcal B(0,R_0))}+\lVert DR\rVert_{L^\infty(\mathcal B(0,R_0))}\right)\left|\varphi(X)-\psi(X)\right| \nonumber  \\
&\leq \left(d^{-\alpha n/4}+4\rho^2\right)\lVert\varphi-\psi\rVert_{L^\infty(\{|X|\leq 2\rho^{60}\})} \nonumber  \\
&\leq  5d^{-\alpha n/4}\, \lVert\varphi-\psi\rVert_{L^\infty(\{|X|\leq 2\rho^{60}\})} \label{eq:first-contraction}\\
&\leq \frac{1}{100}\lVert\varphi-\psi\rVert_{L^\infty(\{|X|\leq 2\rho^{60}\})}, \nonumber 
\end{align}
so $ \|\varphi_{-p} - \psi_{-p}\|_{L^\infty(|X| \leq 20 \rho^{60})} \leq \frac{1}{100} \|\varphi - \psi\|_{L^\infty(|X| \leq 2\rho^{60})}$. In \eqref{eq:first-contraction} we have used that $\rho = d^{-\alpha n \kappa}$ with $\kappa > 1$.

 Let $\widetilde \Gamma_n = \{(X,\varphi_{-n}(X))\}$   and   $\widetilde{\Upsilon}_n = \{(X, \psi_{-n}(X))\}$ ) be the images of $\widetilde \Gamma$  and $\widetilde{\Upsilon}$  under $g_n \circ g_{\ell -1} \circ \cdots g_1 \circ g_0$ respectively. Repeating the above arguments for each factor in the above decomposition yields 
 \begin{equation} \label{eq:graph-contraction}
 \|\varphi_{-n} - \psi_{-n}\|_{L^\infty(|X| \leq 20 \rho^{60})} \leq \Big( \frac{1}{100}\Big)^{\ell +1} \|\varphi - \psi\|_{L^\infty(|X| \leq 2 \rho^{60})}.
 \end{equation}
 
 We now move back to the initial frame $\Ec_0$. By construction $\Gamma = \{ (z,\varphi(z)) \}$ and  $\Upsilon = \{ (z,\psi(z)) \}$ are the images of  $\widetilde \Gamma$ and $\widetilde{\Upsilon}$ under   $\tau_x \circ  C(x)$ and  by the definition of the graph transform $\Gamma_n =  \Lambda(\Gamma)$ is the image of $\widetilde \Gamma_n$ under $\tau_{f^{-n}(x)} \circ C(f^{-n}(x))$, and similarly for $\Upsilon_n$.  Moreover $\overline{\Gamma}_n = \Crm(\widetilde \Gamma_n) $ and $\Gamma_n = \big( \tau_x \circ C(x) \big) (\overline{\Gamma}_n)$, where $\Crm = C^{-1}(x) \circ \tau_x^{-1} \circ \tau_{f^{-n}(x)} \circ C(f^{-n}(x))$ and similarly for $\Upsilon_n$, see the proof of Proposition \ref{prop:graph-transform}.
 
 From the proof of Proposition \ref{prop:graph-transform}, in the frame $\Cc$, the graph $\overline{\Gamma}_n$ is parametrized by $ X \mapsto  X e_1 + \overline\varphi_{-n} (X) e_2$  for $|X| \leq 2 \rho^{60}$,  with $\Lip(\overline\varphi_{-n} ) \leq 3/4$ and $|\overline\varphi_{-n} (0)| \leq \rho^{60}$, and similarly for $\overline{\Upsilon}_n$ with another $\overline\psi_{-n}$ instead of $\overline\varphi_{-n}$ with the same properties.
 
 \smallskip

\noindent \underline{\textit{Claim:}} In the above notation we have   $$ \|\overline\varphi_{-n} - \overline\psi_{-n}\|_{L^\infty(|X| \leq 2 \rho^{60})} \leq 2 \|\varphi_{-n}- \psi_{-n}\|_{L^\infty(|X| \leq 10\rho^{60})}.$$

\smallskip

\textit{Proof of claim:} We will use the notation and results from the proof of Proposition \ref{prop:graph-transform}. In particular,  will use the decomposition \eqref{eq:C-decomposition} and the fact that $W = C(x)^{-1} D   = (W_1,W_2)$ is such that  $\|W \| \leq  \rho^{60} / 5$.

We have that $$\Crm(X,\varphi_{-n}(X)) = \big( (1+\varepsilon_1)X + \varepsilon_3 \varphi_{-n}(X) + W_1, (1+\varepsilon_4)\varphi_{-n}(X) + \varepsilon_2 X + W_2 \big) =: (\overline X_n, \overline Y_n)$$
belongs to $\overline{\Gamma}_n = \{(X, \overline\varphi_{-n}(X))\}$ when $|X| \leq 20\rho^{60}$, that is, $\overline \varphi_{-n}(\overline X_n) =   \overline Y_n$, and similarly for  $\Crm(X,\psi_{-n}(X)) $.

By introducing,
\begin{alignat*}{2}
c_\varphi &\coloneqq (1+\varepsilon_1)X + \varepsilon_3\varphi_{-n}(X) + W_1, &\qquad c_\psi &\coloneqq (1+\varepsilon_1)X + \varepsilon_3\psi_{-n}(X) + W_1, \\
d_\varphi &\coloneqq (1+\varepsilon_4)\overline{\varphi}_{-n}(X) + \varepsilon_2 X + W_2, &\qquad d_\psi &\coloneqq (1+\varepsilon_4)\overline{\psi}_{-n}(X) + \varepsilon_2 X + W_2.
\end{alignat*}
we get
\begin{align*}
\left|\overline{\varphi}_{-n}(\overline{X}_n)-\overline{\psi}_{-n}(\overline{X}_n)\right| &= \left|\overline{\varphi}_{-n}(c_\varphi)-\overline{\psi}_{-n}(c_\varphi)\right| \\
&\leq \left|\overline{\varphi}_{-n}(c_\varphi)-\overline{\psi}_{-n}(c_\psi)\right| + \left|\overline{\psi}_{-n}(c_\psi)-\overline{\psi}_{-n}(c_\varphi)\right| \\
&\leq \left|d_\varphi-d_\psi\right| + |\varepsilon_3|\left|\psi_{-n}(X)-\varphi_{-n}(X)\right| \\
&\leq \bigl(1+|\varepsilon_3+\varepsilon_4|\bigr)\left|\psi_{-n}(X)-\varphi_{-n}(X)\right| \\
&\leq 2\left|\psi_{-n}(X)-\varphi_{-n}(X)\right|.
\end{align*}
thus proving the claim. \hfill $\square$

Combining the above claim with \eqref{eq:graph-contraction} yields $$\|\overline\varphi_{-n} - \overline\psi_{-n}\|_{L^\infty(|X| \leq  \rho^{40})} \leq 2 \Big(\frac{1}{100}\Big)^{\ell +1} \|\varphi - \psi\|_{L^\infty(|X| \leq \rho^{40})} =  2 \Big(\frac{1}{100}\Big)^{\ell +1} d(\Gamma,\Upsilon).$$

Observe that the left hand side above is the distance between $\overline{\Gamma}_n$ and $\overline{\Upsilon}_n$.  Moreover $\Gamma_n = \big( \tau_x \circ C(x) \big) (\overline{\Gamma}_n)$,  $\Upsilon_n = \big( \tau_x \circ C(x) \big) (\overline{\Upsilon}_n)$ and the change of coordinates preserve the graph distance in the corresponsing frames. We conclude that
$$ d(\Lambda(\Gamma),\Lambda(\Upsilon)) = d(\Gamma_n,\Upsilon_n) \leq   2 \Big(\frac{1}{100}\Big)^{\ell +1} d(\Gamma,\Upsilon).$$

This proves the first assertion of Theorem \ref{thm:graph-transform}-(1). By the Banach fixed-point theorem, the map $\Lambda$ admits a fixed point $\Gamma(x)$, since $\Gc_\Sc$ is compact and therefore complete. This completes the proof of Theorem \ref{thm:graph-transform}-(1).

\end{proof}

We now move on the proof of Theorem \ref{thm:graph-transform}-(2). Let $\Gamma(x)$ be the fixed point of $\Lambda$ as above.  The fact that $\Gamma(x)$ is stable under $f^n$ is clear from our construction.  Indeed,  from the definition of the $g_j$'s (see \eqref{eq:def-g_j} and the diagram after Definition \ref{def:frames})  we have
\begin{equation} \label{eq:f^-n-decomposition}
 f^{-n} = \tau_{f^{-n}(x)} \circ C(f^{-n}(x)) \circ g_n \circ g_{\ell -1} \circ \cdots g_1 \circ g_0 \circ C(x)^{-1} \tau_x^{-1}
\end{equation}
and $\Lambda$ is defined as the action of the  map  on the right hand side on the graphs of $\Gc_\Sc$. 
As above,  let  $\widetilde \Gamma_0:= \widetilde \Gamma  = \big(C(x)^{-1} \circ \tau_x^{-1} \big) (\Gamma)$,  $\widetilde \Gamma_p = g_0(\widetilde \Gamma)$,  $\widetilde \Gamma_{2p} = g_1(\widetilde \Gamma_p)$ and so on,  up to $\widetilde \Gamma_{n} = g_n(\widetilde \Gamma_{\ell p})$.

\begin{lemma}
Let $j=0,1,\ldots,\ell-1$.  Let $Z \in \widetilde \Gamma_{jp}$ and $w_j$ be a unit vector in $\Tan_Z \Gamma_{jp}$.  Then $\|D g_j(Z) \cdot w_j\| \geq d ^{\alpha n / 8}$.  A similar estimate holds for $Z \in \widetilde \Gamma_{\ell p}$ under the action of $Dg_n$.
\end{lemma}

\begin{proof} \label{lemma:Dg-tangent}
We only prove the lemma for $j=0$,  the other cases being identical.  Recall that $\widetilde \Gamma_0= \widetilde \Gamma$ is contained in  the disk centered at the origin and radius $\rho^{35}$ (see the proof of Lemma \ref{lemma:g-Gamma}).  In particular $\|Z\| \leq \rho^{35}$.  Using Lemma \ref{lemma:g-estimates} we get that $$\|D g_0(Z) \cdot w_0\| \geq  \|D g_0(0) \cdot w_0\| - \|D g_0(0) \cdot w_0 - D g_0(Z) \cdot w_0\|  \geq \|D g_0(0) \cdot w_0\| - \rho^2.$$

Recall that $\widetilde \Gamma$ is a graph of the form $ \{(X,\varphi (X))\}$ over the disc $\{|X| < 10 \rho^{40}\}$ and $\Lip ( \varphi ) \leq 1 / 2$.  Hence, the unit vector $w_0 \in \Tan_Z \Gamma$ is of the form $w_0 = \frac{(1,\varphi'(X_0))}{\|(1,\varphi'(X_0))\|} $ for some $X_0$ in the above disc.  As $|\varphi'(X_0)| \leq 1/2$ from the above properties of $\varphi$ we see  that  $\|(1,\varphi'(X_0))\| \leq 2$. 
Together with Lemma \ref{lemma:g-estimates} and the above inequality,  we conclude that
\begin{align*}
\|D g_0(Z) \cdot w_0\| & \geq \frac12 \|D g_0(0) \cdot (1,0) + D g_0(0) \cdot (0,\varphi'(X_0))\| \\
& \geq \frac12  \|D g_0(0) \cdot e_1\| - \frac12 |\varphi'(X_0)| \frac12 \|D g_0(0) \cdot (0,1)\| \\
&\geq \frac12 d^{\alpha n /4} -\frac14 d^{-\alpha n /4} \geq d ^{\alpha n / 8},
\end{align*}
thus proving the lemma.
\end{proof}

\begin{proof}[Proof of Theorem \ref{thm:graph-transform}-(2)]
We have already seen that $\Gamma(x)$ is preserved by $f^n$.  We now show that $f^{n}: \Gamma(x) \to \Gamma(x)$ is a contraction. Fix $y \in \Gamma(x)$ and $u \in \Tan_y \Gamma(x)$ of unit norm.  We want to  show that $\|Df^n(y) \cdot u\| \leq d^{-n/100}$.

 Taking the inverse of \eqref{eq:f^-n-decomposition} gives
$$f^{n}
= \tau_x \circ C(x) \circ g_0^{-1} \circ g_1^{-1} \circ \cdots \circ g_{\ell-1}^{-1} \circ g_n^{-1} \circ C\!\left(f^{-n}(x)\right)^{-1} \circ \tau_{f^{-n}(x)}^{-1}.$$
From the chain rule and the fact that  $C(\cdot)$ is linear and $D \tau_\bullet$ is the identity, we get that
$$Df^n(y) \cdot u = C(x) D\big(g_0^{-1} \circ g_1^{-1} \circ \cdots \circ g_{\ell-1}^{-1} \circ g_n^{-1} \big)(y_n)\cdot \widetilde w_n,$$
where $y_n:=C\!\left(f^{-n}(x)\right)^{-1} \circ \tau_{f^{-n}(x)}^{-1}(y)$ and $\widetilde w_n:=C\!\left(f^{-n}(x)\right)^{-1} \cdot u$. Observe that $y_n \in \widetilde \Gamma_n$ and $w_n$ is tangent to $\widetilde \Gamma_n$.
Set $$y_j = g_j^{-1} \circ \cdots \circ g_{\ell-1}^{-1} \circ g_n^{-1}(y_n) \quad \text{for} \quad  j=0,\ldots,\ell - 2 \quad  \text{and} \quad  y_{\ell-1} :=g_n^{-1}(y_n),$$
so that 
$$Df^n(y) \cdot u = C(x) Dg_0^{-1}(y_0) Dg_1^{-1}(y_1) \cdots  Dg_{\ell-1}^{-1}(y_{\ell-1}) D g_n^{-1}(y_n)\cdot \widetilde w_n.$$

Now let $$w_j= \frac{Dg_j^{-1}(y_j) \cdots  Dg_{\ell-1}^{-1}(y_{\ell-1}) D g_n^{-1}(y_n)\cdot \widetilde w_n}{\|Dg_j^{-1}(y_j) \cdots  Dg_{\ell-1}^{-1}(y_{\ell-1}) D g_n^{-1}(y_n)\cdot \widetilde w_n\|} \quad \text{for} \quad  j=0,\ldots,\ell - 1 \quad  \text{and} \quad w_n = \frac{\widetilde w_n}{\|\widetilde w_n\|}.$$

Observe that $y_j \in \widetilde \Gamma_{jp}$ and $w_j \in \Tan_{y_j} \widetilde \Gamma_{jp}$ for all $j=0,\ldots,\ell - 1$.
Then,
\begin{align} \label{eq:Df^n-chain}
\|Df^n(y) \cdot u \ | &\leq \|C(x)\| \|Dg_0^{-1}(y_0) \cdot w_0\| \, \|  Dg_1^{-1}(y_1) \cdots  Dg_{\ell-1}^{-1}(y_{\ell-1}) D g_n^{-1}(y_n)\cdot \widetilde w_n \| \nonumber \\
\cdots \quad & \leq  \|C(x)\| \|Dg_0^{-1}(y_0) \cdot w_0 \| \,  \|Dg_0^{-1}(y_1) \cdot w_1 \| \, \cdots  \,  \|Dg_n^{-1}(y_n) \cdot w_n \| \, \|\widetilde w_n\|
\end{align}

Using Lemma \ref{lemma:Dg-tangent},  we obtain
\begin{align*}
1 &= \|D(g_j \circ g_j^{-1}) (y_j) \cdot w_j)\| = \|Dg_j(g_j^{-1}(y_j)) Dg_j^{-1}(y_j) \cdot w_j\| \\
&=  \Big\|Dg_j(g_j^{-1}(y_j)) \cdot \frac{Dg_j^{-1}(y_j) \cdot w_j}{\|Dg_j^{-1}(y_j) \cdot w_j\|} \Big\| \, \| Dg_j^{-1}(y_j) \cdot w_j \| \\
& \geq d^{\alpha n / 8} \| Dg_j^{-1}(y_j) \cdot w_j \|,
\end{align*} 
so $\| Dg_j^{-1}(y_j) \cdot w_j \| \leq d^{-\alpha n / 8}$,  which combined with \eqref{eq:Df^n-chain} and Lemma \ref{lemma:pesin-matrix} yields
$$\|Df^n(y) \cdot u \ |  \leq (2 \rho^{-6}) (d^{-\alpha n / 8})^{\ell + 1} (2 \rho^{-6}).$$

Now recall that $p=\ceil{6\alpha n}$ and 
$n=(\ell+1)p+s$.  It follows readily from these relations that, for $\alpha>0$ sufficiently small and $n$ sufficiently large one has $
\alpha\ell\geq \frac18$.  Recall that $\rho = d^{-\alpha \kappa}$, so  that
$$\|Df^n(y) \cdot u \ |  \leq 4 d^{12 \alpha  \kappa n} d^{- n / 64} \leq d^{-n/100},$$
because $\alpha $ is small and $n$ is large.  This proves the desired estimate.

The above estimate implies that $f^n: \Gamma(x) \to \Gamma(x)$ is a contraction with respect to the euclidean metric and the existence of a (contracting) periodic point of order  $n$ for $f$ in  $\Gamma(x)$ follows from Banach fixed point theorem.
\end{proof}

\textbf{The reverse graph transform and the unstable multiplier.} The above graph transform can be applied in the reverse direction and,  as a consequence,  it implies that the periodic point obtained above is of saddle type.

\begin{corollary} \label{cor:graph-transform}
 Fix $x \in   \Bc^* \cap \Sc$ and assume that $\dist(x,f^{-n}(x)) < \rho^{70}$ as in Theorem \ref{thm:graph-transform}.  Denote by $q \in  \Bc^* \cap \Sc$ the  fixed point of $f^n$ given by that theorem.  Then for any unit vector $v$ tangent to the plaque $\Delta^{\u}(q)$ of $ \T^-_{r}$ through $q$,  we have that $$\|Df^n(q) v\| \geq d^{n/100}.$$
 In particular,  $q$ is a saddle periodic point of $f$ of period $n$.
\end{corollary}

\section{Periodic points}

In this section we give the main application of the results obtained above: a quantitative equidistribution result for periodic points.  This quantifies the main theorem in \cite{bedford-lyubich-smillie}.

\begin{theorem} \label{thm:periodic-points}
Let $f: \C^2 \to \C^2$ be a Hénon map of degree $d \geq 2$.    Denote by $\mu$ its measure of maximal entropy and by ${\rm SPer_n}$ the set of saddle periodic points of order $n$.  Then, for every  $0<\beta \leq 1$,  there exist a constant  $0<\xi<1$ independent of $\beta$ and another constant $A_\beta>0$ such that the following holds.
$$\Big| \Big\langle\frac{1}{d^{n}}\sum_{a\in {\rm SPer_n}} \delta_a -\mu, \phi \Big\rangle \Big| \leq A_\beta \xi^{\alpha n} \|\phi\|_{\Cc^\beta},$$
for any $\Cc^\beta$ test function $\phi$ on $\C^2$, where $\delta_a$ denotes the Dirac mass at $a$.
\end{theorem}

For the proof,  we will use the quantitative graph transform introduced in the previous section in order to produce a periodic point inside each stable box.  Once this step is achieved,  the proof of the above theorem follows that of the main theorem in \cite{ddk:periodic-Pk} (see Section 6 of that paper).  We keep the notations from the previous section, in particular that of Definition \ref{def:stable-box}.

Fix  $n \geq 1$ as in the previous sections.  As in \cite{ddk:periodic-Pk},  we will use a splitting of the $n^{\text{th}}$ iterate of $f$:
$$f^n = f^{n - n'} \circ f^{n'},  \quad \text{where} \quad n':= \ceil{\ell p/2}$$
and analogously for $f^{-n}$.  Recall that $ p:= \ceil{6 \alpha n}$. This splitting improves the estimates stemming from the exponential mixing property of $f$.  Note that $n'$ is roughly $n / 2$.

We'll need the following standard notion adapted to our context. Observe that every stable box $\Sc$ is equipped with a natural projection $\pi_{\Sc}: \Sc \to Q_{\rho^{60}}$ onto the first coordinate. A subset $E  \subset \Sc$ is said to be \emph{vertical in $\Sc$} if $\pi_{\Sc}(E)$ is compactly contained in $Q_{\rho^{60}}$.

The construction from the previous sections gives uniformly laminar currents $\T^-_{r,\bullet}$ such that, for every stable box $\Sc$ we have
$$\T^-_{r,\bullet}|_{\Sc} = \int_{\Gc_\Sc} [\Delta] d \nu^-_{r,\bullet}(\Delta),$$
for some  positive measure in $\Gc_\Sc$.  In this section we'll only work with $\T^-_{r} = \T^-_{r,0}$,  $\T^-_{r,n'}$ and a refinement of them that we now introduce.
 
 Denote by $\Lc^-_{r}$ the underlying lamination of $\T^-_{r}$, that is, the analytic continuation of the above $\Delta \in \Gc_\Sc$ along the union of all stable boxes $\Sc$.  It is a lamination of the whole vertical strip $\mathcal T$ containing $\Sc$. Similarly, we let  $\Lc^-_{r,n'}$  be the lamination associated with $\T^-_{r,n'}$

Let $$\Gc_{r,\rm{vert}} := \{\Delta \text{ tangent to } \Lc^-_{r} : \Delta \cap \Sc \text{ is vertical in } \Sc \text{ for every stable box }  \Sc \}$$ and define $\Gc_{r,n',\rm{vert}}$ similarly.

We set $ \nu^-_{r,\rm{vert}}:=  \mathbf 1_{\Gc_{r,\rm{vert}}} \nu^-_{r}$, $ \nu^-_{r,n',\rm{vert}}:=  \mathbf 1_{\Gc_{r,n',\rm{vert}}} \nu^-_{r}$  and 
$$\mathbf T^-_{r} := \int_{\Gc_{r,\rm{vert}}} [\Delta] \, d \nu^-_{r,\rm{vert}}(\Delta) \quad \text{and} \quad \mathbf T^-_{r,n'} := \int_{\Gc_{r,\rm{vert}}} [\Delta] \, d \nu^-_{r,n',\rm{vert}}(\Delta).$$

It is clear that $ 0 \leq \mathbf T^-_{r} \leq  \T^-_{r}$ and  $ 0 \leq \mathbf T^-_{r,n'} \leq   \T^-_{r,n'}$.

\begin{lemma} \label{lemma:support-corridor}
Both $ \T^-_{r} -   \mathbf  T^-_{r}$ and  $ \T^-_{r,n'} -   \mathbf  T^-_{r,n'}$ are supported by the union of all vertical corridors and the stable boxes disjoint from $\Bc^*$. 
\end{lemma}

\begin{proof}
Let $y \in \supp \big(  \T^-_{r} -   \mathbf  T^-_{r} \big) $.   By the definition of $ \T^-_{r}$ and $\mathbf  T^-_{r}$ the point $y$ belongs to a holomorphic disc $\Delta^{\u}(y)$ subordinate to $ \T^-_{r}$ that is not vertical inside some stable box $\Sc$. If $\Sc$ is disjoint from $\Bc^*$ there's nothing to show.   Assume otherwise and fix $x \in   \Bc^* \cap \Sc$.  Let $L$ be the vertical line in $\C^2$ passing through $y$ and $B:= L \cap \Sc$ be the portion of $L$ inside $\Sc$.  In the notation of Section \ref{sec:graph-transform}, we have $\Sc = \Psi\big(Q_{\rho^{60}} \times Q^\perp_{ (\rho^{100}/2)^{1 / \theta_-}}  \big)$ and by Proposition \ref{prop:holo-motion} we see that the diameter of $B$ is at most $ \rho^{100}$.

Let $\Delta^{\s}(x)$ be the plaque of  $ \T^+_{r}$ through $x$.  Since $\Delta^{\s}(x)$ is contained in a graph of $\Gc_\Sc$ (Lemma \ref{lemma:plaques-in-G}), it intersects $B$ at a single point $w:= \Delta^{\s}(x) \cap B$. Observe that $\Delta^{\s}(w)  = \Delta^{\s}(x)$ and that $x,y$ and $w$ all belong to $\Sc$, which has diameter less than $\rho^{50}$.  Moreover,   $w$ and $y$ belong to $B$,  so  $\dist(w,y) \leq \rho^{100}$. By  Lemma \ref{lemma:single-point} below,  $\Delta^{\s}(w)$ and $\Delta^{\u}(y) $ intersect at a single point $\mathtt p$ with $\dist(y, \mathtt p) \leq 100 \rho^{-10} \dist(w,y) \leq 100 \rho^{90}$.   Since   $\Delta^{\u}(y)$ is not vertical in $\Sc$, after replacing $y$ by another point $y' \in  \Delta^{\u}(y)$ if necessary (in particular $ \Delta^{\u}(y) =  \Delta^{\u}(y')$ and $y'$ still belongs to the support of  $ \T^-_{r} $),  we can we can assume that the above point $\mathtt p$ is outside $\Sc$.

On one hand,  we have that  $\mathtt p \notin \Sc$ and $\dist(y, \mathtt p) \leq 100 \rho^{90}$.  On the other hand,  the width of the vertical corridor inside $\Sc$ is $r^\gamma \rho^{60}$,  so $y$ must belong to $\pi_D^{-1}\big(Q_{\rho^{60}} \setminus r^\gamma Q_{\rho^{60}} \big)) \cap \Sc$,  the portion of the vertical corridor contained in $\Sc$. 
\end{proof}

\begin{lemma} \label{lemma:single-point}
Let $x,y \in \Bc^*$. Let $\Delta^{\s}(x)$ the plaque of $ \T^+_{r}$ through $x$  and  $\Delta^{\u}(y)$ that of $ \T^-_{r}$ through $y$ respectively.  Assume that $\dist(x,y) \leq \rho^{50}$.  Then $\Delta^{\s}(x)$ and $\Delta^{\u}(y)$ intersect at a single point $\mathtt p$.  Moreover,  $\dist(x, \mathtt p) \leq 100 \rho^{-10} \dist(x,y)$ and $\dist(y, \mathtt p) \leq 100 \rho^{-10} \dist(x,y)$.   The same holds for the intersection of $\Delta^{\u}(x)$ and $\Delta^{\s}(y)$.
\end{lemma}

Recall that $K$ is a large bi-disc in the coordinates $(t,s)$ associated with the splitting $D \oplus D^\perp$ containing the Julia set of $f$.   Let $\Cc$ the union of all vertical and horizontal corridors intersecting $K$ as in Lemma \ref{lemma:dujardin-fubini2}. 

 Fix $x \in \Bc^* \cap (K \setminus \Cc)$.  By the above mentioned lemma and Property (P10) this is a set of large measure.  Let $\Sc = \Psi\big(Q_{\rho^{60}} \times Q_{ (\rho^{100}/2)^{1 / \theta_-}}  \big)$ be a stable box containing $x$ as in Definition \ref{def:stable-box} and $\Gc_\Sc$ the associated space of stable-like graphs (Defintion \ref{def:stable-like-graphs}).
 
 Let $x':= f^{-n'}(x)$. Denote by $$\Sc' = \Psi'\big(Q'_{\rho^{60}} \times Q'_{ (\rho^{100}/2)^{1 / \theta_-}}  \big)$$ the stable box containing $x'$  and $\Gc_{\Sc'}$ the corresponding stable box and space of stable-like graphs as in Defintion \ref{def:stable-like-graphs}.  We'll produce a periodic point inside $\Sc$ by working with the couple  $(\Sc,\Sc')$ and using the exponential mixing.

 Recall that every plaque of $\T^+_{r}$ in $\Sc$ is contained in a graph of $\Gc_\Sc$,  see Lemma \ref{lemma:plaques-in-G}.  The graph transform introduced in Section \ref{sec:graph-transform} was obtained from the successive images of elements $\Gamma \in \Gc_\Sc$ along the composition $g_n \circ g_{\ell -1} \circ \cdots g_1 \circ g_0$.  Denote by $\Lambda'$  the partial graph transform associated with $g_{\ceil{\ell/2}} \circ \cdots g_1 \circ g_0$.  Then $\Lambda'(\Gamma) \cap \Sc' \in \Gc_{\Sc'}$.
 
 \begin{definition} \label{def:good-component-A} Let   $x$ and $n'$ be as above.  We denote  by $\Gc_\Sc(x)$ the subset of $\Gc_\Sc$ consisting of the graphs $\Gamma \in \Gc_\Sc$ containing a  plaque of $\T^+_{r}$ and by $\Gc_\Sc(f^{-n'}(x))$ the subset of $\Gc_{\Sc'}$ consisting of  the graphs $\Gamma'$ containing the cut-off image  $\Lambda'(\Gamma) \cap \Sc'$ for some $\Gamma \in \Gc_\Sc(x)$.   
 Define 
 $$\Ac:= \{ \Gamma' \cap \Sc_{\Sc'} : \Gamma' \in \Gc_\Sc(f^{-n'}(x))\}$$
 and $\widehat \Ac$ be the corresponding connected component of $f^{-n'}(\Sc) \cap \Sc$ after applying $f^{-n'}$.  We call such  $\widehat \Ac$ a \textit{good component}. 
 \end{definition}

  The following lemma is a  version of \cite[Lemma 5]{bedford-lyubich-smillie} that takes into account the splitting $f^n = f^{n - n'} \circ f^{n'}$.  Set
\begin{equation} \label{eq:mu-hat-r}
  \widehat \mu_r:= \Big( \frac{1}{d^{n'}} (f^{n'})^* \T^+_{r} \Big)  \wedge \mathbf T^-_{r,-n'}.
\end{equation}

 \begin{lemma} \label{lemma:bls-mixed}  Let $\Delta^{\s}(\Sc)$ be a fixed plaque  in $\Ac$ and $\Delta^{\u}(\Sc'): = f^{n'}(\Delta^{\u}_0) \cap \Sc$,   where $\Delta^{\u}_0$ is a fixed plaque of $ \mathbf T^-_{r,-n'}$  in $\Sc'$.  Then,   
 \begin{equation*}
 \widehat \mu_r (\, \widehat {\Ac} \, ) =  \frac{1}{d^{n'}}  \int_{\Delta^{\u}(\Sc)}  \T^+_{r}  \cdot \int_{\Delta^{\s}(\Sc')}  \mathbf T^-_{r,-n'} 
 \end{equation*}
 \end{lemma}
 
 \begin{proof}
 Observe that any such  $\Delta^{\u}(\Sc')$ intersects any plaque of $ \T^+_{r} \,$ graph of $\Gc_{r,n',\rm{vert}}$ in a single point.  Moreover,  for any horizontal $(1,1)$-current $S$,  the intersection $S \wedge \Delta^{\u}(\Sc)$ is well defined and its mass is independent of  $\Delta^{\u}(\Sc)$.
 By the definition of $ \mathbf T^-_{r,-n'} $,  the left hand side of the inequality we want to obtain is equal to
 \begin{align*}
 \frac{1}{d^{n'}}  \int_{\Gc_{r,n',\rm{vert}}} \int_{\Delta^{\u} \cap \widehat \Ac } \, (f^{n'})^* \T^+_{r} \, d \nu^-_{r,n',\rm{vert}}(\Delta^{\u}) = \frac{1}{d^{n'}}   \int_{\Gc_{r,n',\rm{vert}}} \int_{f^{n'} (\Delta^{\u} \cap \widehat \Ac ) } \ \T^+_{r} \, d \nu^-_{r,n',\rm{vert}}(\Delta^{\u})
 \end{align*}
 
 From the above remarks,  the intersection $ \T^+_{r} \wedge [f^{n'} (\Delta^{\u} )]$ is independent of $\Delta^{\u}$ and equals $ \T^+_{r} \wedge [\Delta^{\u} (\Sc)]$ , so the last integral equals
 \begin{align*}
  \frac{1}{d^{n'}}  \int_{ \Delta^{\u}(\Sc) } \ \T^+_{r}  \cdot  \int_{\Gc_{r,n',\rm{vert}}} \, d \nu^-_{r,n',\rm{vert}}(\Delta^{\u}).
 \end{align*}
 
 On the other hand, 
 $$ \int_{\Gc_{r,n',\rm{vert}}} \, d \nu^-_{r,n',\rm{vert}}(\Delta^{\u}) = \int_{\Gc_{r,n',\rm{vert}}} \int [\Delta^{\s}(\Sc') \cap \Delta^{\u}] \, d \nu^-_{r,n',\rm{vert}}(\Delta^{\u}) =  \int_{\Delta^{\s}(\Sc')}  \mathbf T^-_{r,-n'} $$
 where we have  used that $\Delta^{\s}(\Sc') \cap \Delta^{\u}$ is a single point.  This finishes the proof. 
 \end{proof}

\begin{lemma} \label{lemma:mu-r-hat-mass}
Let $ \widehat \mu_r $ be as above.  Then $ \widehat \mu_r \leq   \mu $ and $ \| \mu -  \widehat \mu_r \| \leq C r^{\gamma }$ for some constant $C >0$.
\end{lemma}

\begin{proof}
The first assertion is clear,  because $$\widehat \mu_r = \Big( \frac{1}{d^{n'}} (f^{n'})^* \T^+_{r} \Big)  \wedge \mathbf T^-_{r,-n'} \leq  \Big( \frac{1}{d^{n'}} (f^{n'})^* T^+ \Big)  \wedge  T^- =  T^+    \wedge  T^- = \mu.$$

In order to prove the second assertion,  we  write  $\mu -  \widehat \mu_r = \nu_a +  \nu_a + \nu_c$,
where $ \nu_a :=  \mu - \mu_{r,-n'}, $
$$ \nu_b:= \mu_{r,-n'} - \Big( \frac{1}{d^{n'}} (f^{n'})^* \T^+_{r} \Big)  \wedge T^-_{r,-n'}  \quad \text{and} \quad \nu_c :=  \Big( \frac{1}{d^{n'}} (f^{n'})^* \T^+_{r} \Big)  \wedge ( T^-_{r,-n'} -\mathbf T^-_{r,-n'}) $$
and let $\mathbf A := \|\nu_a\|$,  $\mathbf B := \|\nu_b\|$ and $\mathbf C := \|\nu_c\|$.  Then $\|\mu -  \widehat \mu_r \| =  \mathbf A +  \mathbf B +  \mathbf C$.  We estimate each term separately.   

 The term $\mathbf A$ is bounded by $C_1 r^\gamma$ for some constant $C_1 >0$ by Property (P10).  By Lemma \ref{lemma:support-corridor}, the measure $\nu_c$ is supported by $\Cc \cup (\Bc^*)^c$.  Since $\nu_c \leq \mu$,  Lemma \ref{lemma:dujardin-fubini2} and Property (P11) imply that $\mathbf C$ is bounded by $C_2 r^\gamma$ for some constant $C_2 > 0$.
 
 It remains to estimate $\mathbf B$.  We have, 
 \begin{align*}
 \nu_b &=  \Big(  \T^+_{r,-n'}  - \frac{1}{d^{n'}} (f^{n'})^* \T^+_{r} \Big)  \wedge T^-_{r,-n'} \leq \Big(  T^+  - \frac{1}{d^{n'}} (f^{n'})^* \T^+_{r} \Big)  \wedge T^-_{r,-n'} \\
 &=\frac{1}{d^{n'}}   (f^{n'})^*   \Big(  T^+  - \T^+_{r} \Big)  \wedge T^-_{r,-n'} \leq  \frac{1}{d^{n'}}   (f^{n'})^*   \Big(  T^+  - \T^+_{r} \Big)  \wedge T^-
 \end{align*}
 
 Integrating both sides,  we get that $\mathbf B$ is bounded by the mass of $(T^+  - \T^+_{r}) \wedge  \frac{1}{d^{n'}}   (f^{n'})_* T^-$  
$=(T^+  - \T^+_{r}) \wedge  T^-$ which is bounded  $C_3 r^{\gamma}$ for some constant $C_3 > 0$ (see the proof of Proposition \ref{prop:geometric-intersection0}). This finishes the proof of the lemma.
\end{proof}

\begin{definition}
Fix $x \in \Bc^* \cap (K \setminus \Cc)$ and $x' = f^{-n'}(x)$ and let $\Sc$ and $\Sc'$ be the corresponding stable boxes as above.  A  \textit{good components} of $f^{-n'}(\Sc') \cap \Sc$ 

We denote  by $\mathcal A_{n'}(\Sc,\Sc')$ the set  of  good components of $f^{-n'}(\Sc') \cap \Sc$ as in Definition \ref{def:good-component-A} and let $b_{\Sc,\Sc'}(n')$ be its cardinality.  
\end{definition}

\begin{proposition} \label{prop:mixing-upper}
Let $\Sc$ and $\Sc'$  be as above.  Then
$$\mu(f^{-n'}(\Sc) \cap \Sc') \leq b_{\Sc,\Sc'}(n')  \frac{1}{d^{n'}}  \int_{\Delta^{\u}(\Sc)}  \T^+_{r}  \cdot \int_{\Delta^{\s}(\Sc')}  \mathbf T^-_{r,-n'}  + m_{\Sc,\Sc'}(n'),$$
where $m_{\Sc,\Sc'}(n') $ are non-negative numbers depending on $\Sc$ and $\Sc'$.  Moreover,  $$\sum_{\Sc,\Sc'} m_{\Sc,\Sc'}(n')  \leq C r^\gamma,$$ where the (finite) sum runs over all pairs of stable boxes intersecting $K$.  The statement holds after replacing $n'$ by $n-n'$, $\T^+_{r}$ by $\T^+_{r,-n'}$, $\mathbf T^-_{r,-n'}$ by $\mathbf T^-_{r} $  and permuting $\Sc$ and $\Sc'$.
\end{proposition}

\begin{proof}
We first split
$$\mu(f^{-n'}(\Sc) \cap \Sc') = \sum_{\widehat \Ac \in \mathcal A_{n'}(\Sc,\Sc')} \mu(\widehat \Ac) + \mu\big( (f^{-n'}(\Sc) \cap \Sc') \setminus \cup_{\widehat \Ac \in \mathcal A_{n'}(\Sc,\Sc')} \widehat \Ac  \,  \big).$$

Write the summand in first term as
$$
\mu(\widehat \Ac)  =  \widehat \mu_r(\widehat \Ac) + \big( \mu(\widehat \Ac)  -  \widehat \mu_r  (\widehat \Ac) \big),
$$
where $\widehat \mu_r$ is defined in \eqref{eq:mu-hat-r}.

Using Lemma \ref{lemma:bls-mixed} and summing over $\widehat \Ac$,   we get that 
$$\mu(f^{-n'}(\Sc) \cap \Sc') \leq b_{\Sc,\Sc'}(n')  \frac{1}{d^{n'}}  \int_{\Delta^{\u}(\Sc)}  \T^+_{r}  \cdot \int_{\Delta^{\s}(\Sc')}  \mathbf T^-_{r,-n'}  + m_{\Sc,\Sc'}(n'), $$
where 
$$ m_{\Sc,\Sc'}(n'):= \sum_{\widehat \Ac \in \mathcal A_{n'}(\Sc,\Sc')} \big( \mu(\widehat \Ac)  -  \widehat \mu_r  (\widehat \Ac) \big) +  \mu\big( (f^{-n'}(\Sc) \cap \Sc') \setminus \cup_{\widehat \Ac \in \mathcal A_{n'}(\Sc,\Sc')} \widehat \Ac  \,  \big).$$

It remains to estimate the sum of $m_{\Sc,\Sc'}(n')$ over the pairs $(\Sc,\Sc')$.  By the definition of the sets $\widehat \Ac$,  every point of $f^{-n'}(\Sc \cap \Bc^*) \cap \Sc'$ belongs to some good component $\widehat \Ac$,  so the set  $(f^{-n'}(\Sc) \cap \Sc') \setminus \cup_{\widehat \Ac \in \mathcal A_{n'}(\Sc,\Sc')} \widehat \Ac$ is included in $f^{-n'}(\Bc^*)$.  Together with the invariance of $\mu$ and  the fact that the sets $f^{-n'}(\Sc) \cap \Sc'$ are pairwise disjoint,  we obtain that the sum over the pairs $(\Sc,\Sc')$ of the second term in the definition of $m_{\Sc,\Sc'}(n')$ is bounded by $\mu((\Bc^*)^c) \leq C_1 r^\gamma$.  For the remaining sum,  we observe that that the $\widehat \Ac$'s and the pairs $f^{-n'}(\Sc) \cap \Sc'$ are all pairwise disjoint,  so the desired open bound follows from the mass estimate of Lemma \ref{lemma:mu-r-hat-mass}.
\end{proof}

The following notions are the analogue of the ones \cite{ddk:periodic-Pk} adapted to our setting.

\begin{definition} \label{def:nice}
We say that $\Sc$ is \textit{nice} with respect to $\Sc'$ at order $n'$ if $$m_{\Sc,\Sc'}(n') \leq r^{\gamma / 2} \mu(f^{-n'}(\Sc) \cap \Sc').$$
Otherwise,  we say that  $\Sc$ is bad with respect to $\Sc'$.  The same definition applies to $n- n'$ instead of  $n'$ by switching the roles of $\Sc$ and $\Sc'$.
We say that a couple $(\Sc,\Sc')$ is \textit{good} if $\Sc$ is nice with respect to $\Sc'$ at order $n'$ \textit{and} $\Sc$ is nice with respect to $\Sc'$ at order $n-n'$.
\end{definition}

\begin{proposition} \label{prop:good-couple}
Let $(\Sc,\Sc')$ be a good couple.  Then,  $$b_{\Sc,\Sc'}(n')  b_{\Sc',\Sc}(n - n') \geq (1-r^{\gamma / 2})^2 d^n \big( \mu(\Sc) \mu(\Sc') - \varepsilon_{\Sc,\Sc'} (n)\big),$$
where $\varepsilon_{\Sc,\Sc'} (n)$ are non negative numbers such that $\sum_{\Sc,\Sc'} \varepsilon_{\Sc,\Sc'} (n) \leq C r^\gamma$ for some constant $C>0$.
\end{proposition}

Before proving the above proposition, we need a couple of intermediary lemmas.  

In order to apply the exponential mixing for Hölder continuous observables, one first needs a good cut-off function supported by the stable boxes.  Recall that $\Sc = \Psi\big(Q_{\rho^{60}} \times Q^\perp_{ (\rho^{100}/2)^{1 / \theta_-}}  \big)$ over which we defined vertical and horizontal corridors (see Definition \ref{def:stable-box}).  The following lemma can be proved using a standard cut-off function of class $\Cc^1$ over $Q_{\rho^{60}} \times Q^\perp_{ (\rho^{100}/2)^{1 / \theta_-}}$ and pushing it to $\Sc$ via $\Psi$. The estimates follow from the ones of the initial cut-off function together with the regularity estimates of $\Psi$ (see Proposition \ref{prop:holo-motion} and the paragraph that follows). We leave the details to the reader. 

\begin{lemma} \label{lemma:cut-off-function}
 Denote by $\mathring{\Sc}$ the set obtained from $\Sc$  by removing its corridors.  There exists a function $\varphi_\Sc$ on $\C^2$ of class $\Cc^{\theta_-}$ such that $0\leq \varphi_\Sc \leq 1$,  $\varphi_\Sc$ is supported on $\Sc$,   $\varphi_\Sc \equiv 1$ on $\mathring{\Sc}$ and $\|\varphi_\Sc\|_{\Cc^{\theta_-}} \leq 8 \rho^{10}$.
\end{lemma}

The following result is a consequence of the exponential mixing for Hölder observables due to the second named author \cite{dinh:mixing}.

\begin{lemma} \label{lemma:mixing-lower}
There exists non-negative numbers $e_{\Sc,\Sc'}(n') $  depending on $\Sc$ and $\Sc'$ such that 
$$\mu(f^{-n'}(\Sc) \cap \Sc') \geq \mu(\Sc) \mu(\Sc') - e_{\Sc,\Sc'}(n')$$
and  $\sum_{\Sc,\Sc'} e_{\Sc,\Sc'}(n')  \leq C r^\gamma$.
\end{lemma}

\begin{proof}
Let $\varphi_\Sc$  and $\varphi_{\Sc'}$  be the cut-off functions associated with $\Sc$  and $\Sc'$  respectively given by Lemma \ref{lemma:cut-off-function}. Then, by the lemma and the main theorem in \cite{dinh:mixing}, we get
\begin{align*}
\mu(f^{-n'}(\Sc) \cap \Sc')  &\geq \int (\varphi_\Sc \circ f^{n'}) \,  \varphi_{\Sc'} \, d \mu \\ &\geq \Big( \int \varphi_\Sc \, d \mu \Big)  \Big( \int \varphi_{\Sc'} \, d \mu \Big) - c d^{-n' \theta_-^2 / 8 } \|\varphi_\Sc\|_{\Cc^{\theta_-}} \, \|\varphi_{\Sc'}\|_{\Cc^{\theta_-}}
\end{align*}

Denote by  $\Cc_\Sc:= \Sc \setminus \mathring{\Sc}$ the corridors inside $\Sc$ and similarly for $\Sc'$. Then $\int \varphi_\Sc \, d \mu  \geq \mu(\mathring{\Sc}) =   \mu( \Sc) -  \mu(\Cc_\Sc)$ and and similarly for $\Sc'$. Hence, the last quantity above is bounded from below by
\begin{align*}
 \big( \mu(\Sc) -  \mu(\Cc_\Sc) \big) \big( \mu(\Sc') -  \mu(\Cc_{\Sc'}) \big) - c d^{-n' \theta_-^2 / 8 } \|\varphi_\Sc\|_{\Cc^{\theta_-}} \, \|\varphi_{\Sc'}\|_{\Cc^{\theta_-}}
\end{align*}

Expanding the above product and setting
$$e_{\Sc,\Sc'}(n') : = \mu(\Cc_\Sc) \mu(\Sc') +  \mu(\Cc_{\Sc'}) \mu(\Sc) + c d^{-n' \theta_-^2 / 8 } \|\varphi_\Sc\|_{\Cc^{\theta_-}} \, \|\varphi_{\Sc'}\|_{\Cc^{\theta_-}}  $$
we get that
$$\mu(f^{-n'}(\Sc) \cap \Sc') \geq \mu(\Sc) \mu(\Sc') - e_{\Sc,\Sc'}(n').$$

It remains to bound the sum of $e_{\Sc,\Sc'}(n')$ over the pairs $(\Sc,\Sc')$. For the last term in the definition of $e_{\Sc,\Sc'}(n')$, notice that, by the fact that $K$ has finite area,  the number of stable boxes intersecting $K$ is bounded by a constant times $\rho^{-60} (\rho^{-100}/2)^{-1/\theta_-}$ which is smaller than $\rho^{-300}$ because $\theta_- > 1/2$. This fact, together  with Lemma \ref{lemma:cut-off-function} yield
 $$ \sum_{\Sc,\Sc'} c \, d^{-n' \theta_-^2 / 8 } \|\varphi_\Sc\|_{\Cc^{\theta_-}} \, \|\varphi_{\Sc'}\|_{\Cc^{\theta_-}}  \leq  c' \, d^{-n'  / 32 } \rho^{-20} \rho^{-600} \leq c' \,  d^{- 3n / 128} d^{620 \alpha \kappa n} \leq c' \,  d^{-\alpha n \gamma} = c' \, r^\gamma$$
for some constant $c' >0$, where  we have used that $\theta_- \geq \frac12$,  $n$ is large,  $\alpha$ is small and $
\alpha\ell\geq \frac18$ as noted before, so $n' \geq 6 \alpha n \frac{\ell}{2} \geq  \frac{3}{4} n.$

For the remaining term, we have
\begin{align*}
 \sum_{\Sc,\Sc'} \Big(  \mu(\Cc_\Sc) \mu(\Sc') +  \mu(\Cc_{\Sc'}) \mu(\Sc) \Big) &= \sum_\Sc \mu(\Cc_\Sc) \sum_{\Sc'} \mu(\Sc') +    \sum_{\Sc'}  \mu(\Cc_{\Sc'})  \sum_\Sc  \mu(\Sc) \\ &= 2 \sum_\Sc \mu(\Cc_\Sc) \sum_\Sc \mu(\Sc),
\end{align*}
but  $\sum_\Sc  \mu(\Sc)$ is smaller than one because $\mu$ is of mass one, and $\sum_{\Sc}  \mu(\Cc_\Sc)$ is the mass of the union of all corridors intersecting $K$, which is smaller than $C r^\gamma$ by Lemma \ref{lemma:dujardin-fubini2}. This finishes the proof of the lemma.
\end{proof}

\begin{remark}
The above lemma is an example of application of the exponential mixing where one really needs to work with Hölder continuous test functions, the $\Cc^1$ observables not being enough.
\end{remark}

We can now prove the proposition. 

\begin{proof}[Proof of Proposition 5.11]
Set
$$c_{\Sc,\Sc'}(n'):= \frac{\mu(f^{-n'}(\Sc) \cap \Sc') - m_{\Sc,\Sc'}(n')}{ \int_{\Delta^{\u}(\Sc)}  \T^+_{r}  \cdot \int_{\Delta^{\s}(\Sc')}  \mathbf T^-_{r,-n'} }$$
and
$$c_{\Sc',\Sc'}(n-n'):= \frac{\mu(f^{-n+n'}(\Sc') \cap \Sc) - m_{\Sc',\Sc}(n')}{ \int_{\Delta^{\u}(\Sc')}  \T^+_{r,-n'}  \cdot \int_{\Delta^{\s}(\Sc)}  \mathbf T^-_{r}}.$$
Proposition \ref{prop:mixing-upper}, yields
\begin{equation*}
b_{\Sc,\Sc'}(n')  b_{\Sc',\Sc}(n - n') \geq d^n \, c_{\Sc,\Sc'}(n') \, c_{\Sc',\Sc}(n-n').
\end{equation*}

We have that  $\int_{\Delta^{\u}(\Sc)}  \T^+_{r}  \cdot \int_{\Delta^{\s}(\Sc')}  \mathbf T^-_{r,-n'} \leq \mu(\Sc)$ and $ \int_{\Delta^{\u}(\Sc')}  \T^+_{r,-n'}  \cdot \int_{\Delta^{\s}(\Sc)}  \mathbf T^-_{r} \leq \mu(\Sc')$. Using the assumption that the couple $(\Sc,\Sc')$ is good, and Lemma \ref{lemma:mixing-lower}, we obtain
\begin{align*}
& c_{\Sc,\Sc'}(n') \,  c_{\Sc',\Sc}(n-n')  \geq (1-r^{\gamma / 2})^2  \frac{\mu(f^{-n'}(\Sc) \cap \Sc') \cdot \mu(f^{-n+n'}(\Sc') \cap \Sc)}{\mu(\Sc) \mu(\Sc')} 
\\ & \geq (1-r^{\gamma / 2})^2  \frac{\big( \mu(\Sc) \mu(\Sc')- e_{\Sc,\Sc'}(n') \big)\cdot \big( \mu(\Sc) \mu(\Sc') - e_{\Sc',\Sc}(n-n')\big) }{\mu(\Sc) \mu(\Sc')}
\\ &=  (1-r^{\gamma / 2})^2  \bigg[ \mu(\Sc) \mu(\Sc')- e_{\Sc,\Sc'}(n') - e_{\Sc',\Sc}(n-n') + \frac{e_{\Sc',\Sc}(n-n')\big) }{\mu(\Sc) \mu(\Sc')} \bigg]
\\ & \geq  (1-r^{\gamma / 2})^2  \big[ \mu(\Sc) \mu(\Sc')- e_{\Sc,\Sc'}(n') - e_{\Sc',\Sc}(n-n') \big].
\end{align*}
and  $\sum_{\Sc,\Sc'} \varepsilon_{\Sc,\Sc'} (n) \leq C r^\gamma$. 
The proposition follows by setting $ \varepsilon_{\Sc,\Sc'} (n):= e_{\Sc,\Sc'}(n') + e_{\Sc',\Sc}(n-n')$.
\end{proof}

The following notions are also adapted versions of the ones in \cite{ddk:periodic-Pk}  to our setting.

\begin{definition}
Let $\varepsilon_{\Sc,\Sc'} (n)$ be as in the statement of Proposition \ref{prop:good-couple}. We say that a stable box is \textit{non-admissible} if $$\sum_{\substack{\Sc' \\ (\Sc,\Sc') \text{ is a good couple}}} \varepsilon_{\Sc,\Sc'} (n) \geq r^{\gamma /2} \mu(\Sc).$$
Otherwise,  we say that $\Sc$ is \textit{admissible}. 

We say that $\Sc$  is \textit{safe} if it is admissible and $$\sum_{\substack{\Sc' \\ (\Sc,\Sc') \text{ is a good couple}}} \mu(\Sc') \geq 1 - r^{\gamma/4}.$$
\end{definition}

Observe that Proposition \ref{prop:good-couple} gives
\begin{equation} \label{eq:non-admissible}
\sum_{\substack{\Sc \\ \Sc \text{ is non-admissible}}} \mu(\Sc) \leq r^{-\gamma / 2} \sum_{\Sc,\Sc'} \varepsilon_{\Sc,\Sc'} (n) \leq C r^{\gamma / 2}.
\end{equation}

In particular,  this will allows  to discard all non-admissible cells in the proof of Theorem \ref{thm:periodic-points}.  The following proposition shows that the unsafe cells can also be discarded.

\begin{proposition} \label{prop:unsafe}
In the above notation we have that
$$\sum_{\substack{\Sc \\ \Sc \text{ is not safe}}} \mu(\Sc) \leq C r^{\gamma /4}.$$
\end{proposition}

\begin{proof}
Thanks to \eqref{eq:non-admissible} it is enough to bound
$$\sum_{\substack{\Sc \text{admissible}\\ \Sc \text{ is not safe}}} \mu(\Sc).$$

Let $\Sc$ be admissible and not safe. We can write
$$ \mu(\Sc) = \sum_{\substack{\Sc' \\ (\Sc,\Sc') \text{ good}}}  \mu(\Sc)  \mu(\Sc') +  \sum_{\substack{\Sc' \\ (\Sc,\Sc') \text{ not good}}}  \mu(\Sc)  \mu(\Sc').$$

Since $\Sc$ is not safe, the first sum above is bounded from above by $(1-r^{\gamma / 4}) \mu(\Sc)$. It follows that
$$r^{\gamma/4}  \mu(\Sc)  \leq \sum_{\substack{\Sc' \\ (\Sc,\Sc') \text{ not good}}}  \mu(\Sc)  \mu(\Sc'), \quad \text{so} \quad  \mu(\Sc)  \leq r^{-\gamma / 4}  \sum_{\substack{\Sc' \\ (\Sc,\Sc') \text{ not good}}}  \mu(\Sc)  \mu(\Sc').$$

It follows that
$$\sum_{\substack{\Sc \text{admissible}\\ \Sc \text{ is not safe}}} \mu(\Sc) \leq r^{-\gamma /4} \sum_{\Sc} \sum_{\substack{\Sc' \\ (\Sc,\Sc') \text{ not good}}}  \mu(\Sc)  \mu(\Sc')$$
and we are lead to show that

\begin{equation} \label{eq:sum-not-good}
 \sum_{\Sc} \sum_{\substack{\Sc' \\ (\Sc,\Sc') \text{ not good}}}  \mu(\Sc)  \mu(\Sc') \leq C r^{\gamma / 2}.
\end{equation}

In order to show that, denote by   $\mathrm{BC}$ the set of  bad couples. Then, the above sum equals

\begin{align*}
\mu \otimes \mu \bigg( \bigcup_{(\Sc, \Sc') \in \mathrm{BC}}  \Sc \times \Sc' \bigg) & \leq \mu \otimes \mu \bigg( \bigcup_{(\Sc, \Sc') \in \mathrm{BC}_{n'}}  \Sc \times \Sc' \bigg) + \mu \otimes \mu \bigg( \bigcup_{(\Sc, \Sc') \in \mathrm{BC}_{n'-n}}  \Sc \times \Sc' \bigg)
\end{align*}
where we define $\mathrm{BC}_{n'}$ as the set of couples $(\Sc, \Sc')$ where $\Sc$ is not nice with $\Sc'$ at order $n'$, and similarly for $\mathrm{BC}_{n'-n}$, as in  Definition \ref{def:nice}.

If $\Sc$ is not nice with $\Sc'$ at order $n'$ we have  $m_{\Sc,\Sc'}(n') \geq r^{\gamma / 2} \mu(f^{-n'}(\Sc) \cap \Sc')$ by definition. Together with Lemma \ref{lemma:mixing-lower}, we obtain
$$m_{\Sc,\Sc'}(n') \geq r^{\gamma / 2} \big(\mu(\Sc) \mu(\Sc') - e_{\Sc,\Sc'}(n') \big),$$
so
$$ r^{-\gamma / 2} m_{\Sc,\Sc'}(n') +  e_{\Sc,\Sc'}(n')  \geq  \mu(\Sc) \mu(\Sc').$$
A parallel argument applies to $m_{\Sc',\Sc}(n-n')$.
Therefore,
\begin{align*}
\mu \otimes \mu \bigg( \bigcup_{(\Sc, \Sc') \in \mathrm{BC}}  \Sc \times \Sc' \bigg) \leq \sum_{\Sc, \Sc'} \big( r^{-\gamma / 2} m_{\Sc,\Sc'}(n') +  e_{\Sc,\Sc'}(n')  \big) \\ + \sum_{\Sc', \Sc} \big( r^{-\gamma / 2} m_{\Sc',\Sc}(n-n')  +  e_{\Sc',\Sc}(n - n')  \big)
\end{align*}

By Proposition \ref{prop:mixing-upper} and Lemma \ref{lemma:mixing-lower}, we obtain \eqref{eq:sum-not-good}. This finishes the proof of the proposition.
\end{proof}

Once the unsafe cells are discarded,  the last step is to show that each safe cells contains a good number of periodic points with an exponential error.  This is the content of the next result.

\begin{proposition} \label{prop:safe}
Let $\Sc$ be a fixed safe cell.  Denote by ${\rm SPer_n}(\Sc)$ the set of saddle periodic points of order $n$ belonging to $\Sc$. Then $$\# {\rm SPer_n}(\Sc) \geq (1 - r^{\gamma / 4}) d^n \mu(\Sc)$$
\end{proposition}

\begin{proof}
Recall that $b_{\Sc,\Sc'}(n')$ is the number of  good components of $f^{-n'}(\Sc') \cap \Sc$ as in Definition \ref{def:good-component-A} and  $b_{\Sc',\Sc}(n-n')$ is the number of  good components of $f^{-n+n'}(\Sc') \cap \Sc$.

Let $\Sc$ be a safe cell.  The graph transform theorem and the closing lemma (see  Theorem \ref{thm:graph-transform} and Corollary \ref{cor:graph-transform}) give that 
$$\# {\rm SPer_n}(\Sc) \geq \sum_{\substack{\Sc' \\ (\Sc,\Sc') \text{ good}}}b_{\Sc,\Sc'}(n') \, b_{\Sc',\Sc}(n-n') $$
and by Proposition \ref{prop:good-couple}, this quantity is at least
$$\sum_{\substack{\Sc' \\ (\Sc,\Sc') \text{ good}}} b_{\Sc,\Sc'}(n')  \, b_{\Sc',\Sc}(n - n') \geq (1-r^{\gamma / 2})^2 d^n \big( \mu(\Sc) \mu(\Sc') - \varepsilon_{\Sc,\Sc'} (n)\big).$$

Since is safe (and in particular admissible), we have $$\sum_{\substack{\Sc' \\ (\Sc,\Sc') \text{ good}}} \mu(\Sc') \geq 1 - r^{\gamma/4} \quad \text{and} \sum_{\substack{\Sc' \\ (\Sc,\Sc')  \text{ good}}} \mu(\Sc') \leq r^{\gamma /2} \mu(\Sc).$$ 

Combining these estimates gives $\# {\rm SPer_n}(\Sc) \geq (1 - r^{\gamma / 4}) d^n \mu(\Sc)$ and proves the proposition.
\end{proof}

\begin{proof}[Proof of Theorem \ref{thm:periodic-points}]
In view of Propositions \ref{prop:unsafe} and \ref{prop:safe},  the proof proceeds, \textit{mutatis mutandis,} exactly as in Section 6 of \cite{ddk:periodic-Pk}.
\end{proof}

\bibliographystyle{alpha}
\bibliography{refs}

\end{document}